\documentclass[12pt,oneside,reqno]{amsart}
      \usepackage{amsmath}
      \usepackage{amssymb}
      \usepackage[all]{xy}
      \usepackage{arcs}
      \usepackage[utf8x]{inputenc}
      \usepackage{graphicx}
      \usepackage{color}
      \usepackage{epsf,amsfonts}
      \usepackage{graphics}

      \theoremstyle{plain}
      \newtheorem{theorem}{Theorem}[section]
      \newtheorem{lemma}[theorem]{Lemma}
      \newtheorem{corollary}[theorem]{Corollary}

      \newtheorem{definition}[theorem]{Definition}    

      \theoremstyle{construction}
      \newtheorem{construction}[theorem]{Construction}
      \theoremstyle{example}
      \newtheorem{example}[theorem]{Example}
      \theoremstyle{remark}
      \newtheorem{remark}[theorem]{Remark}      
      \numberwithin{equation}{section}

      \newcommand{\C}{\mathbb{C}}
      \newcommand{\F}{\mathbb{F}}

      \newcommand{\nil}{\varnothing}

      \newcommand{\kk}{\mathbf{k}}
      \newcommand{\ovl}{\overline}
      \newcommand{\wh}{\widehat}
      \newcommand{\sss}{\vspace{2.5 mm}}
      \newcommand{\dd}{\mathit{DD}\bigl(\textstyle \frac{\mathbb{I}}{2} \bigr)}
      \newcommand{\calD}{\mathcal{D}}
      
      \newcommand{\calN}{\mathcal{N}}
      \newcommand{\calP}{\mathcal{P}}
      \newcommand{\calM}{\mathcal{M}}
      \newcommand{\calZ}{\mathcal{Z}}
      \newcommand{\calA}{\mathcal{A}}
      \newcommand{\calB}{\mathcal{B}}
      \newcommand{\calC}{\mathcal{C}}
      \newcommand{\D}{\mathbb{D}}
      \newcommand{\bdy}{\partial}
      \newcommand{\MCG}{\text{MCG}}
      \newcommand{\Hom}{\text{Hom}}
      \newcommand{\Mor}{\text{Mor}}
      \newcommand{\nn}{\nonumber}
      \newcommand{\im}{\text{im}}
      
      \newcommand{\Dom}{\text{Dom}}

      \newcommand{\simd}{\text{EQ}_{\bdy\boldsymbol{D}}}
      \newcommand{\simt}{\text{EQ}_{\bdy\boldsymbol{T}}}
      \newcommand{\init}{\text{init}}
      \newcommand{\term}{\text{term}}
      \newcommand{\tot}{\text{tot}}
      \newcommand{\Int}{\text{Int}}
      
      \newcommand{\BS}{\mathrm{BS}}
      \newcommand{\TP}{\text{TP}}
      \newcommand{\IP}{\text{IP}}
      \newcommand{\tx}{\tilde{x}}
      \newcommand{\tf}{\tilde{f}}
      \newcommand{\tg}{\tilde{g}}
      \newcommand{\unl}{\underline}
      \newcommand{\ova}{\overarc}
      \newcommand{\una}{\underarc}
      \newcommand{\DD}{\boldsymbol{D}}
      \newcommand{\TT}{\boldsymbol{T}}
      \newcommand{\aoa}{\vphantom{\mathbb{I}}^{\mathcal{A}}\mathbb{I}_{\mathcal{A}}}
      \newcommand{\GEN}{\text{GEN}}
      \newcommand{\executeiffilenewer}[3]{%
      \ifnum\pdfstrcmp{\pdffilemoddate{#1}}%
      {\pdffilemoddate{#2}}>0%
      {\immediate\write18{#3}}\fi%
      }
      
\begin{document}

       \author{Kyler Siegel}
      \address{Department of Mathematics, Stanford University, Palo Alto, CA 94305}
       \email{ksiegel@stanford.edu}
  
       \title[]{A Geometric Proof of a Faithful Linear-Categorical Surface Mapping Class Group Action}
 \begin{abstract}
We give completely combinatorial proofs of the main results of [\ref{faithful}] using polygons. Namely, we prove that the mapping class group of a surface with boundary acts faithfully
on a finitely-generated linear category. Along the way we prove some foundational results 
regarding the relevant objects from bordered Heegaard Floer homology, 

 \end{abstract}
    \maketitle
\tableofcontents
\section{Introduction}
An important open question in topology is whether the mapping class group of a surface with boundary is linear.
In [\ref{faithful}], the authors use the tools of bordered Heegaard Floer homology to answer a categorifed version of this question,
showing that the mapping class group of a surface with boundary $F$ acts faithfully on
a finitely-generated linear category. Specifically, the category is the derived category of finitely-generated left $\calB(F)$-modules,
where $\calB(F)$ is an algebra associated to the surface $F$.
There is also a standard bimodule $\dd$, and a version $\boxtimes$ of the tensor product, which lets us define a bimodule $\dd \boxtimes M(\phi)$.
The action comes from assigning to $\phi$ the functor 
\begin{align*}
(\dd \boxtimes M(\phi)) \otimes_{\calB(F)}\cdot,
\end{align*}
and then passing to the derived category.
The bimodule $M(\phi)$ can be defined in terms of curves on $F$ and polygons formed between them, and this gives very concrete geometric interpretations
to the results involved. Incidentally, $M(\phi)$ is not actually an ordinary bimodule but a more general $\calA_{\infty}$-bimodule,
and checking $\calA_{\infty}$ identities tends to involve verifiying infinitely many relations.

While [\ref{faithful}] gives combinatorial definition of the bimodules $M(\phi)$, it refers proofs about their structure to [\ref{Bimodules},\ref{LOT10c}], which rely on hard analysis.
In this paper we give completely combinatorial proofs of the main results in [\ref{faithful}] in terms of polygons and operations on them.
In light of [\ref{Auroux}], it is interesting to compare our results with [\ref{Abouzaid}].

The paper is organized as follows. In Section \ref{some definitions} we define some of the basic objects that will be used throughout the paper,
especially $\calA_{\infty}$-algebras and modules, and we define $\calB(F)$ and $M(\phi)$.
In Section \ref{Some Preliminary Results} we prove some foundational results:
\begin{theorem}
$M(\phi_0)$ is a $\calB(\calZ')$-$\calB(\calZ)$ $\calA_{\infty}$-bimodule.
\end{theorem}
\begin{theorem}
Let $\phi_b$ and $\phi_{nb}$ be diffeomorphisms representing the same element of $\MCG_0(F)$. Then 
$M(\phi_b)$ and $M(\phi_{nb})$ are homotopy equivalent as $\calA_{\infty}$-bimodules.
\end{theorem}
In Section \ref{The Main Result} we give the key ingredient for an action:
\begin{corollary}
The bimodules $M(\psi\circ\phi)$ and $M(\phi) \boxtimes \left(\dd \boxtimes M(\psi)\right)$ are $\calA_{\infty}$-homotopy equivalent.
\end{corollary}
We then complete the proof by showing that the identity mapping class gives the identity functor and that the action is faithful:
\begin{theorem}\label{identityisidentity}
$\dd \boxtimes M(\mathbb{I})$ is isomorphic as a type DA bimodule to $\aoa$.
\end{theorem}
\begin{theorem}
If $\dd \boxtimes M(\phi_0)$ is quasi-isomorphic to $\aoa$, then $\phi_0$ is isotopic to $\mathbb{I}$.
\end{theorem}

\subsection{Acknowledgements}
I would like to thank my advisor Robert Lipshitz for his guidance and devotion to teaching me this material.  This project was only possible because of 
his patience and ability to elucidate difficult mathematics.

\section{Some Definitions}\label{some definitions}

In this section we give definitions for some of the objects that will play key roles in this paper. 

\subsection{$\calA_{\infty}$-algebras and $\calA_{\infty}$-bimodules}\label{algebras and bimodules}

We begin by defining $\calA_{\infty}$-algebras, and $\calA_{\infty}$-bimodules over 
$\calA_{\infty}$-algebras. 
These are like associative algebras and ordinary differential bimodules except with associativity replaced by a weaker condition involving infinitely many terms.
Actually all of our $\calA_{\infty}$-bimodules will be defined over ordinary differential algebras, but we define $\calA_{\infty}$-algebras here for completeness.
We fix a ground ring $\kk$, which for our purposes will always be a direct sum of copies of $\mathbb{F}_2$.

\begin{definition}
An $\calA_{\infty}$-algebra $\calA$ over $\kk$ is a $\kk$-module $A$, together with $\kk$-linear maps
\begin{align*}
\mu_i: A^{\otimes i} \rightarrow A
\end{align*}
for each $i \geq 1$, satisfying the following structure equation for each $n \geq 1$ and $a_1,...,a_n \in A$:
\begin{align}\label{ainfalgebrarelations}
\sum_{i+j = n+1}\sum_{l=1}^{n-j+1}\mu_i(a_1\otimes ... \otimes a_{l-1} \otimes \mu_j(a_l \otimes ... \otimes a_{l+j-1}) \otimes a_{l+j} \otimes ... \otimes a_n) = 0. 
\end{align}
\end{definition}

For example, the structure equation for $n=3$ can be interpreted as
\begin{align*}\tiny
\xymatrix@=1em
{
a_1\ar@/_/@{->}[dr] && a_2\ar@{->}[dd] && a_3\ar@/^/@{->}[ddll] & & a_1\ar@/_/@{->}[ddrr] && a_2\ar@{->}[d] && a_3\ar@/^/@{->}[ddll] & & a_1\ar@/_/@{->}[ddrr] && a_2\ar@{->}[dd] && a_3\ar@/^/@{->}[dl] & &  \\
& \mu_1\ar@/_/@{->}[dr] &&&       &+& &&\mu_1\ar@{->}[d]&&         &+& &&&\mu_1\ar@/^/@{->}[dl]&         &+&  \\
&&\mu_3\ar@{->}[d]&&         & & &&\mu_3\ar@{->}[d]&&         & & &&\mu_3\ar@{->}[d]&&         & &  \\
&&&&         & & &&&&         & & &&&&         & &  \\
}\\ \tiny
\xymatrix@=1.5em
{
a_1\ar@/_/@{->}[dr] & a_2\ar@{->}[d] & a_3\ar@/^/@{->}[ddl] && a_1\ar@/_/@{->}[ddr] & a_2\ar@{->}[d] & a_3\ar@/^/@{->}[dl] && a_1\ar@/_/@{->}[dr] & a_2\ar@{->}[d] & a_3\ar@/^/@{->}[dl] &&\\
    & \mu_2\ar@{->}[d] &   &+&  &\mu_2\ar@{->}[d] &           &+&    &\mu_3\ar@{->}[d]&& =&0&\\
    & \mu_2\ar@{->}[d] &   &&  &\mu_2\ar@{->}[d]& &&  &\mu_1\ar@{->}[d] & &&\\
    &       &   &&  &&    && && &&
}\\
\end{align*}

By setting $\mu_0 \equiv 0$ we can combine the $\mu_i$'s to define a single map 
\begin{align*}
\mu := \sum_{i=0}^{\infty}\mu_i: T^*(A) \rightarrow A
\end{align*}
on the tensor algebra 
\begin{align*}
T^*(A) := \bigoplus_{n=0}^{\infty}A^{\otimes n}. 
\end{align*}
Defining $\ovl{D}^{\calA}: T^*(A) \rightarrow T^*(A)$ by 
\begin{align*}
\ovl{D}^{\calA}(a_1\otimes ... \otimes a_n) := \sum_{j=1}^n\sum_{l=1}^{n-j+1}a_1\otimes ... \otimes a_{l-1}\otimes \mu_j(a_l\otimes ... \otimes a_{l+j-1}) \otimes a_{l+j} \otimes ... \otimes a_n,
\end{align*}
the $\calA_{\infty}$-algebra relations $(\ref{ainfalgebrarelations})$ can be written more succinctly as $\mu\circ \ovl{D}^{\calA} = 0$.

Define $\Delta_2: T^*(A) \rightarrow T^*(A) \otimes T^*(A)$ by 
\begin{align*}
\Delta_2(a_1 \otimes ... \otimes a_n) := \sum_{i=0}^n(a_1 \otimes ... \otimes a_i) \otimes (a_{i+1} \otimes ... \otimes a_n).
\end{align*}

\begin{definition}
Let $\calA$ and $\calB$ be $\calA_{\infty}$-algebras over $\kk$. Then an $\calA_{\infty}$-bimodule $_\calA M_{\calB}$ over $\calA$ and $\calB$ is a $\kk$-bimodule $M$ together with
$\kk$-linear maps
\begin{align*}
m_{i,1,j}: A^{\otimes i}\otimes M \otimes B^{\otimes j} \rightarrow M 
\end{align*}
for each $i,j \geq 0$ satisfying the following:
 \begin{align}\label{ainfbimodulerelations}
\xymatrix@C=1.25em@R=2em
{
\ar@{=>}[d]&\ar@{-->}[dd]&\ar@{=>}[d]&&\ar@{=>}[d]&\ar@{-->}[dd]&\ar@/^/@{=>}[ddl]&&\ar@/_/@{=>}[ddr]&\ar@{-->}[dd]&\ar@{=>}[d]&&\\
\Delta_2\ar@/_/@{=>}[dr]\ar@/_/@{=>}[ddr]&&\Delta_2\ar@/^/@{=>}[dl]\ar@/^/@{=>}[ddl]&&\ovl{D}^{\calA}\ar@/_/@{=>}[dr]&&&&&&\ovl{D}^{\calB}\ar@/^/@{=>}[dl]&&\\
&m\ar@{-->}[d]&&+&&m\ar@{-->}[d]&&+&&m\ar@{-->}[d]&&=0.\\
&m\ar@{-->}[d]&&&&&&&&&&&\\
&&&&&&&&&&&&\\
}
\end{align}
\end{definition}
We will also refer to $\calA_{\infty}$-bimodules as {\em type AA bimodules}.

\subsection{Arc diagrams and the algebra $\calB(\calZ)$}\label{arc diagrams and the algebra B(Z)}

Let $F$ be a connected, oriented surface of genus $g$ with $b > 0$ boundary components $Z_1,...,Z_b$, with each $Z_i$ divided into two closed arcs $S_i^+$ and $S_i^-$ which overlap
only at their endpoints. Consider a collection of pairwise-disjoint, embedded paths $\alpha_1,...,\alpha_{2(g+b-1)}$ in $F$ with $\bdy \alpha_i \subset \cup_i S_i^-$, such that $F \setminus (\cup_i \alpha_i)$ is a union
of disks, and the boundary of each such disk contains exactly one $S_i^+$. Assume we also have a basepoint $z_i$ in each $S_i^+$.

Putting $\{a_i,a_i'\} = \bdy \alpha_i$, we call the tuple
\begin{align*}
\calZ := ((Z_1,...,Z_b),(\{a_1,a_1'\},...,\{a_{2(g+b-1)},a_{2(g+b-1)}'\}),(z_1,...,z_b)) 
\end{align*}
an {\em arc diagram} for $F$.

Associated to $\calZ$ we have an algebra $\calB(\calZ)$ over $\kk = \bigoplus_{i = 1}^{2(g+b-1)}\F_2$ with a basis 
over $\mathbb{F}_2$ consisting of one element $I_i$ for each pair $\{a_i,a_i'\}$ and one element $\rho$ for each chord, i.e. nontrivial interval
in $Z_i \setminus \{z_i\}$ with endpoints in $\{a_1,a_1',...,a_{2(g+b-1)},a_{2(g+b-1)}'\}$. For each chord $\rho$ we denote the initial and terminal points (with respect to the 
orientation of $Z_i$) by $\rho^-$ and $\rho^+$ respectively. 

The product on $\calB(\calZ)$ is defined as follows:
\begin{itemize}
 \item The $I_i$'s are orthogonal idempotents, i.e. $I_iI_j = \delta_{ij}I_i$.
 \item $I_i\rho = \rho$ if $a_i$ or $a_i'$ is $\rho^-$, and $I_i\rho = 0$ otherwise. Similarly, $\rho I_j = \rho$ if $a_j$ or $a_j'$ is $\rho^+$, and 
$\rho I_j = 0$ otherwise.
 \item For chords $\rho$ and $\rho'$, $\rho \rho'$ is the chord from $\rho^-$ to $(\rho')^+$ if $\rho^+ = (\rho')^-$ and $\rho\rho' = 0$ otherwise.
\end{itemize}
Note that the sum of the idempotents $\mathbf{1}_{\calB(\calZ)} := \sum_i I_i$ provides a multiplicative identity for $\calB(\calZ)$.

From $\calZ$ we can build a surface $F^\circ(\calZ)$ by thickening each boundary circle $Z_i$ of $\calZ$ to an annulus $Z_i \times [0,1]$ and then 
attaching 2-dimensional 1-handles to each pair $\{a_i,a_i'\}$ in the inner boundaries of the annuli. The surface $F(\calZ)$ obtained by cutting along each 
$z_i \times [0,1] \subset Z_i \times [0,1]$ is canonically identified with $F$ (up to isotopy).

Associated to each $\calZ$ is a dual arc diagram $\calZ'$ coming from $\{\eta_i\cap \bdy F(\calZ)\}$, where $\{\eta_i\}$ is the (unique up to isotopy) dual set of curves in $F^\circ(\calZ)$
such that
\begin{itemize}
 \item $\eta_i$ is contained in the handle of $F^\circ(\calZ)$ corresponding to $\alpha_i$ and
 \item $\eta_i$ intersects $\alpha_i$ in a single point.
\end{itemize}
From $\calZ'$ we get an algebra $\calB(\calZ')$.

\subsection{The bimodule $M(\phi)$}

We now wish to define a $\calB(\calZ')-\calB(\calZ)$ $\calA_{\infty}$-bimodule $M(\phi)$ associated to each mapping class $\phi\in \MCG_0(F(\calZ))$. Here
$\MCG_0(F)$ is the group of isotopy classes of diffeomorphisms of $F$ which fix the boundary of $F$ pointwise.
Let $\phi_0$ be a representative of $\phi$, and let $\phi_0$ act on the $\eta$-curves $\eta_1,...,\eta_{2(g+b-1)}$ of $F^\circ(\calZ)$, giving a new set of curves $\beta_1,...,\beta_{2(g+b-1)}$. 
We put $\boldsymbol{\alpha} = \alpha_1\cup...\cup \alpha_{2(g+b-1)}$ and $\boldsymbol{\beta} = \beta_1\cup...\cup\beta_{2(g+b-1)}$ and let
\begin{align*}
\calD(\phi_0) := (F^\circ(\calZ),\boldsymbol{\alpha},\boldsymbol{\beta}).
\end{align*} 
We can always assume that $\boldsymbol{\alpha}$ and $\boldsymbol{\beta}$ intersect transversally. 

\begin{figure}
 \centering
 \includegraphics[scale=.5]{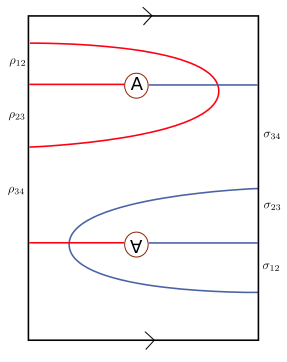}
 \caption{The identity diagram $\calD(\mathbb{I})$.}
\end{figure}

In order to define $M(\phi)$, we will first need the notion of polygons in $\calD(\phi_0)$. Let
\begin{align*}
\D^2 := \{x + iy \in \C\;|\; x \geq 0, x^2 + y^2 \leq 1\}\\
\gamma_L := \bdy\D^2\cap\{x + iy\in\C\;|\;x=0\}\\
\gamma_R := \bdy\D^2\cap\{x + iy\in\C\;|\;x\geq 0\}.
\end{align*}
Note the non-standard definition of $\D^2$. 
We orient $\gamma_L$ and $\gamma_R$ from $-i$ to $i$ and refer to them as the left and right sides of $\D^2$ respectively.
\begin{definition}
Let $\sigma_1,...,\sigma_m$ and $\rho_1,...,\rho_n$ be chords in $\calZ$ and $-\calZ'$ (the orientation reversal of $\calZ'$) respectively,
and let $x,y \in \boldsymbol{\alpha}\cap\boldsymbol{\beta}$ be intersection points in $\calD(\phi_0)$. A {\em polygon in $\calD(\phi_0)$ connecting $x$ to $y$ through
$(\sigma_1,...,\sigma_m)$ and $(\rho_1,...,\rho_n)$} is a map $P: \D^2 \rightarrow \calD(\phi_0)$ such that
\begin{itemize}
 \item $P(\gamma_L) \subset (\boldsymbol{\beta} \cup \bdy\calD(\phi_0))$ and $P(\gamma_R) \subset (\boldsymbol{\alpha}\cup\bdy\calD(\phi_0))$.
 \item There are ordered points $p_1,...,p_{2m} \in \gamma_L$ and $q_1,...,q_{2n} \in \gamma_R$ appearing in order as one traverses $\gamma_L$ and $\gamma_R$ so that 
$P$ is an orientation-preserving immersion on $\D^2\setminus\{p_1,...,p_{2m},q_1,...,q_{2n}\}$.
 \item $P(-i) = x$ and $P(i) = y$.
 \item For each $i$, $P([p_{2i+1},p_{2i+2}]) = \sigma_i$ and $P([q_{2i+1},q_{2i+2}]) = \rho_i$, and except for these intervals, $P$ maps to the interior of $\calD(\phi_0)$.
\end{itemize}
We allow $m$ or $n$ (or both) to be zero. A polygon with $m=n=0$ is a {\em bigon}.
\end{definition}
We will sometimes refer to $P|_{\gamma_L}$ and $P|_{\gamma_R}$ (or their images) as the left and right sides of $P$ respectively.
By abuse of notation we will often not distinguish between a polygon and its image when no confusion will arise.

We are only interested in equivalence classes of polygons, where we define $P$ to be equivalent to $P \circ w$ for any diffeomorphism $w: \D^2 \rightarrow \D^2$.
From now on by polygon we mean equivalence class of polygons.
Let 
\begin{align*}
n(x,y,(\sigma_1,...,\sigma_m),(\rho_1,...,\rho_n)) \in \F_2
\end{align*}
denote the mod two number of equivalence classes of polygons connecting $x$ to $y$
through $(\sigma_1,...,\sigma_m)$ and $(\rho_1,...,\rho_n)$.

We are now finally in a position to define the $\calB(\calZ')-\calB(\calZ)$ $\calA_{\infty}$-bimodule $M(\phi_0)$. 
\begin{definition}
As an $\F_2$-vector space, $M(\phi_0)$ is generated by the intersection points $\boldsymbol{\alpha}\cap\boldsymbol{\beta}$.
For sequences of chords $\sigma_1,...,\sigma_m \in \calB(\calZ')$ and $\rho_1,...,\rho_n \in \calB(\calZ)$ and generators $x,y \in \boldsymbol{\alpha}\cap\boldsymbol{\beta}$ we define
\begin{align*}
m_{m,1,n}(\sigma_m,...,\sigma_1,x,\rho_1,...,\rho_n) := \sum_{y\in\boldsymbol{\alpha}\cap\boldsymbol{\beta}}n(x,y,(\sigma_1,...,\sigma_m),(\rho_1,...,\rho_n))y. 
\end{align*}

For an idempotent $J \in \calB(\calZ')$, we define $m_{1,1,0}(J,x)$ to be $x$ if $J$ is the idempotent corresponding to the $\beta$ arc containing $x$, and 
zero otherwise. Similarly, for an idempotent $I \in \calB(\calZ)$, we define $m_{0,1,1}(x,I)$ to be $x$ if $I$ is the idempotent corresonding to the $\alpha$ arc containing $x$,
zero otherwise. For $m + n > 1$, we define $m_{m,1,n}$ to be zero if any of the inputs are idempotents.
\end{definition}
\begin{remark}
Because of the last paragraph above, we say that $m$ is {\em strictly unitial}.
\end{remark}

Since the $\calA_{\infty}$-homotopy type of $M(\phi_0)$ depends only on the isotopy class of $\phi_0$ (see Section \ref{Homotopy invariance of M under isotopies} below), we will write
$M(\phi) = M(\phi_0)$.

\section{Some Preliminary Results}\label{Some Preliminary Results}
In this section we establish two fundamental facts about $M(\phi)$. In Section $\ref{Checking the A infinity relations for M}$ 
we show that the definition given in the previous section
does indeed make $M(\phi_0)$ into an $\calA_{\infty}$-bimodule, and in Section $\ref{Homotopy invariance of M under isotopies}$ we show that $M(\phi_0)$ is $\calA_{\infty}$-homotopy 
invariant under isotopies
of $\phi_0$ (this is made precise below), and therefore $M(\phi)$ is well-defined up to $\calA_{\infty}$-homotopy equivalence.

\subsection{Checking the $A_{\infty}$ relations for $M(\phi_0)$}\label{Checking the A infinity relations for M}

\begin{theorem}\label{MIsAInfinity}
$M(\phi_0)$ is a $\calB(\calZ')$-$\calB(\calZ)$ $\calA_{\infty}$-bimodule.
\end{theorem}
\begin{proof}
Since $\calB(\calZ')$ and $\calB(\calZ)$ are in fact ordinary associative algebras (a much simplified case
of $\calA_{\infty}$-algebras), many of the terms in the structure equations drop out.
In fact, the $\calA_{\infty}$ relations (\ref{ainfbimodulerelations}) can be written as 
\begin{align}\label{ainfinityrelations}
 0 = \sum_{k = 0}^i\sum_{l=0}^j m(\sigma_i,...,\sigma_{k+1},m(\sigma_k,...,\sigma_1,x,\rho_1,...,\rho_l),\rho_{l+1},...,\rho_j)\nn\\\;\;\;\;\;
 +\; \sum_{k=1}^{i-1}m(\sigma_i,...,\sigma_{k+2},\mu_2(\sigma_{k+1},\sigma_{k}),\sigma_{k-1},...,\sigma_1,x,\rho_1,...,\rho_j)\nn\\\;\;\;\;\;
 +\; \sum_{l=1}^{j-1}m(\sigma_i,...,\sigma_1,x,\rho_1,...,\rho_{l-1},\mu_2(\rho_l,\rho_{l+1}),\rho_{l+2},...,\rho_j)
\end{align}
for each $x \in M(\phi_0),\;\sigma_1,...,\sigma_i \in \calB(\calZ'),\;\rho_ 1,...,\rho_j \in \calB(\calZ)$.

To establish (\ref{ainfinityrelations}), we will show that the nontrivial summands in the expansion of the right side cancel pairwise.
We first introduce some notation.
Let $S_1^L$ denote the set of polygons $\D^2 \rightarrow \calD(\phi_0)$ starting at $x$ through $(\sigma_1,...,\sigma_{k-1},\mu_2(\sigma_{k+1},\sigma_k),\sigma_{k+2},...,\sigma_i)$
and $(\rho_1,...,\rho_j)$ for some $1 \leq k \leq i-1$, and let
$S_1^R$ denote the set of polygons $\D^2 \rightarrow \calD(\phi_0)$ starting at $x$ through $(\sigma_1,...,\sigma_i)$ and $(\rho_1,...,\rho_{l-1},\mu_2(\rho_l,\rho_{l+1}),\rho_{l+2},...,\rho_j)$
for some $1 \leq l \leq j-1$.
For a polygon $P$, let $\IP(P)$ and $\TP(P)$ denote its initial point and terminal point respectively, 
let $\Dom(P)$ denote the domain\footnote{In this paper the word ``domain'' is always synonymous with ``source''; this differs from the terminology elsewhere in the Heegaard Floer literature.}
of $P$, and let $\gamma_L(\Dom(P))$ and $\gamma_R(\Dom(P))$
denote the left and right sides of $\Dom(P)$ respectively.
Let $S_2$ denote the set of pairs $(P_1,P_2)$ of polygons $\D^2 \rightarrow \calD(\phi_0)$,
where $P_1$ is a polygon starting at $x$ through $(\sigma_1,...,\sigma_k)$ and $(\rho_1,...,\rho_l)$
and $P_2$ is a polygon starting at $\TP(P_1)$ through $(\sigma_{k+1},...,\sigma_i)$ and $(\rho_{l+1},...,\rho_j)$,
for some $0 \leq k \leq i,\;0\leq l \leq j$.
Set 
\begin{align*}
S := S_1^L \cup S_1^R \cup S_2. 
\end{align*}
It will suffice to construct a fixed point free involution $I: S \rightarrow S$.

To begin, suppose $(P_1,P_2) \in S_2$. Since $\TP(P_1) = \IP(P_2)$, 
there is an embedding $s_1: [0,1] \rightarrow \bdy\Dom(P_1)$ with $s_1(0) = i$ and an embedding
$s_2: [0,1] \rightarrow \bdy\Dom(P_2)$ with $s_2(0) = -i$, such that $P_1 \circ s_1 = P_2 \circ s_2$.
We choose $s_1$ and $s_2$ to be maximal in length.
Let 
\begin{align*}
P_G: \Dom(P_1) \cup_{s_1 \sim s_2} \Dom(P_2) \rightarrow \calD(\phi_0)
\end{align*}
denote the map obtained by gluing together $P_1$ 
and $P_2$ along $s_1$ and $s_2$.
Let $p \in \Dom(P_G)$ denote the point $s_1(1) = s_2(1)$.
Then $P_G(p)$ is either an intersection point of an $\alpha$ arc and a $\beta$ arc, or else
an endpoint of an $\alpha$ or $\beta$ arc.
In the former case, $p$ is either $i \in \Dom(P_2)$ or $-i \in \Dom(P_1)$. 
We will define $I$ in steps as follows:
\begin{itemize}
 \item {\em Step 1} defines $I((P_1,P_2))$ for $(P_1,P_2) \in S_2$ when $p = i \in \Dom(P_2)$.
 \item {\em Step 2} defines $I((P_1,P_2))$ for $(P_1,P_2) \in S_2$ when $p = -i \in \Dom(P_1)$.
 \item {\em Step 3} defines $I((P_1,P_2))$ for $(P_1,P_2) \in S_2$ when $P_G(p)$ is an intersection point of two adjacent short chords.
 \item {\em Step 4} defines $I(P)$ for $P \in S_1^L \cup S_1^R$. 
\end{itemize}

\noindent {\em Step 1}: Note that $p$ lies either on a) the left side of $\Dom(P_1)$ or b) the right side of $\Dom(P_1)$.
In case a) (resp. b)) there is an embedding $s_3: [0,1] \rightarrow \Dom(P_G)$ with $s_3((0,1)) \subset \Int(\Dom(P_G)),\; s_3(0) = p$, 
and $s_3(1) \in \gamma_L(\Dom(P_1)) \setminus \gamma_L(\Dom(P_2))$ 
(resp. $s_3(1) \in \gamma_R(\Dom(P_1)) \setminus \gamma_R(\Dom(P_2))$),
such that $P_G \circ s_3$ lies along an $\alpha$ arc (resp. $\beta$ arc).
We now observe that $s_3$ cuts $\Dom(P_G)$ into two pieces, and the restrictions of $P_G$ to these pieces
define polygons $P_1'$ from $\IP(P_1)$ to $P_G(s_3(1))$ and $P_2'$ from $P_G(s_3(1))$ to $\TP(P_2)$.
One easily checks that $(P_1',P_2') \in S_2$ and $(P_1',P_2') \neq (P_1,P_2)$, and we define $I((P_1,P_2)) := (P_1',P_2')$.
See Figure \ref{AInfinityGluings}.

\sss

\noindent {\em Step 2}: Similar to Step 1, $p$ lies either a) on the left side of $\Dom(P_2)$ or b) the right side of $\Dom(P_2)$.
In case a) (resp. b)) there is an embedding $s_3: [0,1] \rightarrow \Dom(P_G)$ with 
$s_3((0,1)) \subset \Int(\Dom(P_G)),\; s_3(0) = p$, and $s_3(1) \in \gamma_L(\Dom(P_2)) \setminus \gamma_L(\Dom(P_1))$
(resp. $s_3(1) \in \gamma_R(\Dom(P_2)) \setminus \gamma_R(\Dom(P_1))$),
such that $P_G \circ s_3$ lies along an $\alpha$ arc (resp. $\beta$ arc).
We observe that $s_3$ cuts $\Dom(P_G)$ into two pieces, and the restrictions of $P_G$ to these pieces 
define polygons $P_1'$ from $\IP(P_1)$ to $P_G(s_3(1))$ and $P_2'$ from $P_G(s_3(1))$ to $\TP(P_2)$.
Again, one easily checks that $(P_1',P_2') \in S_2$ and $(P_1',P_2') \neq (P_1,P_2)$, and we define $I((P_1,P_2)) := (P_1',P_2')$.
See Figure \ref{AInfinityGluings}.

\sss

\noindent {\em Step 3}: In this case $p$ lies on either a) the left side of $\Dom(P_1)$ or b) the right side of $\Dom(P_1)$.
In case a) $P_G$ itself defines a polygon $I((P_1,P_2))$ in $S_1^L$, whereas in case b) $P_G$ defines a polygon $I((P_1,P_2))$ in $S_1^R$.

\sss

\noindent {\em Step 4}: Let $P \in S_1^L$ be a polygon through $(\sigma_1,...,\sigma_{k-1},\mu_2(\sigma_{k+1},\sigma_k),\sigma_{k+2},...,\sigma_i)$
and $(\rho_1,...,\rho_j)$. Let $p \in \gamma_L(P)$ be the point corresponding to $\sigma_k \cap \sigma_{k+1}$.
There must be an embedding $s_3: [0,1] \rightarrow \Dom(P)$ with $s_3((0,1)) \subset \Int (\Dom(P)),\;s_3(0) = p$, and $s_3(1) \in \gamma_R(\Dom(P))$.
We can assume that $s_3(1)$ lies between $(\rho_1,...,\rho_l)$ and $(\rho_{l+1},...,\rho_j)$.
Then $s_3$ cuts $\Dom(P)$ into two pieces, and the restrictions of $P$ to these pieces define polygons
$P_1'$ from $\IP(P)$ to $P(s_3(1))$ through $(\sigma_1,...,\sigma_k)$ and $(\rho_1,...,\rho_l)$
and $P_2'$ from $P(s_3(1))$ to $\TP(P)$ through $(\sigma_{k+1},...,\sigma_i)$ and $(\rho_{l+1},...,\rho_j)$.
We define $I(P) := (P_1',P_2') \in S_2$. 

$I(P) \in S_2$ is defined similarly when $P \in S_1^R$.

\sss

We have now defined $I$ on all of $S$. It follows immediately that $I$ is fixed point free, 
and we leave it to the reader to verify that $I$ is an involution.
\end{proof}
\begin{figure}
\begin{center}
 \includegraphics[scale=.6]{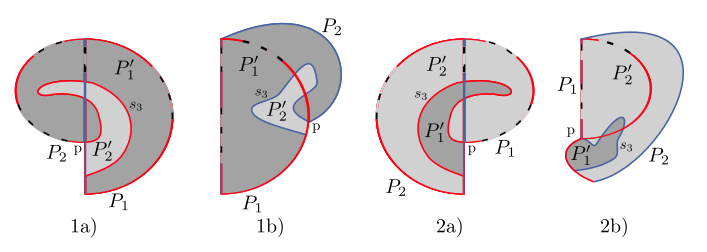}
\end{center}
\caption{Illustrating prototypical examples of {\em Step 1} and {\em Step 2} in the proof of Theorem \ref{MIsAInfinity}. The shadings represent the resulting polygons $P_1'$ and $P_2'$.}
\label{AInfinityGluings}
\end{figure}

\subsection{Homotopy invariance of $M(\phi_0)$ under isotopies}\label{Homotopy invariance of M under isotopies}
Recall that for ordinary chain complexes $(M,d)$ and $(M',d')$, a chain map is a linear map $f: M \rightarrow M'$ such that 
$d'\circ f = f \circ d$, i.e. $f$ commutes with the differential. Two chain maps $f,f': M \rightarrow M'$ are called homotopic
if there exist a linear map $h: M \rightarrow M'$ such that $f - f' = d'\circ h + h \circ d$.
Defining a differential $\bdy$ on linear maps $h: M \rightarrow M'$ by $\bdy(h) = d'\circ h + h\circ d$, we see that a linear map $f$ is a chain map if $\bdy(f) = 0$, and 
linear maps $f$ and $f'$ are homotopic if $f - f' \in \im(\bdy)$.
Observe that $(\Hom(M,M'),\bdy)$ is itself a chain complex, since we have
\begin{align*}
(\bdy \circ \bdy)(h) = \bdy(d'\circ h + h\circ d) = d'\circ d' \circ h + d' \circ h \circ d + d' \circ h \circ d + h \circ d \circ d = 0. 
\end{align*}
Recall that a chain map $f$ is called a homotopy equivalence if there exists a chain map 
$g: M' \rightarrow M$ such that $g \circ f$ is homotopic to $\mathbb{I}_M$ and $f \circ g$ is homotopic to $\mathbb{I}_{M'}$.
 
We now proceed to define analogous terms for the $\calA_{\infty}$ case, which will allow us to make precise the statement that $M(\phi_0)$ is homotopy invariant under 
isotopies of $\phi_0$.
Let $\calA$ and $\calB$ be $\calA_{\infty}$-algebras over $\kk$ and let $_\calA M_{\calB}$ and $_\calA M'_{\calB}$ be $\calA_{\infty}$-bimodules over $\calA$ and $\calB$.
The analogues of linear maps are called {\em morphisms}.
\begin{definition}
A {\em morphism} $\vphantom{M}_\calA M_{\calB} \rightarrow \vphantom{M'}_\calA M'_{\calB}$ is a collection of 
$\kk$-linear maps
\begin{align*} 
f_{i,1,j}: A^{\otimes i} \otimes M \otimes B^{\otimes j} \rightarrow M'
\end{align*}
for $i,j \geq 0$.
\end{definition}

Let $f$ be the total map $f: T^*(A) \otimes M \otimes T^*(B) \rightarrow M'$. With the notation of Section \ref{algebras and bimodules}, we define a differential on morphisms by
\begin{align}\label{A infinity homomorphism relations}
\xymatrix@=.80em
{
&&\ar@{=>}[d]&\ar@{-->}[dd]&\ar@{=>}[d]&&\ar@{=>}[d]&\ar@{-->}[dd]&\ar@{=>}[d]&&\ar@{=>}[d]&\ar@{-->}[dd]&\ar@/^/@{=>}[ddl]&&\ar@/_/@{=>}[ddr]&\ar@{-->}[dd]&\ar@{=>}[d]&&\\
&&\Delta_2\ar@/_/@{=>}[dr]\ar@/_/@{=>}[ddr]&&\Delta_2\ar@/^/@{=>}[dl]\ar@/^/@{=>}[ddl]&&\Delta_2\ar@/_/@{=>}[dr]\ar@/_/@{=>}[ddr]&&\Delta_2\ar@/^/@{=>}[dl]\ar@/^/@{=>}[ddl]&&\ovl{D}^{\calA}\ar@/_/@{=>}[dr]&&&&&&\ovl{D}^{\calB}\ar@/^/@{=>}[dl]&&\\
\bdy(f)&= &&f\ar@{-->}[d]&&+&&m\ar@{-->}[d]&&+&&f\ar@{-->}[d]&&+&&f\ar@{-->}[d]&\\
&&&m\ar@{-->}[d]&&&&f\ar@{-->}[d]&&&&&&&&&.&&\\
&&&&&&&&&&&&&&&&&&\\
}
\end{align}
This makes $(\Mor(\vphantom{M}_\calA M_{\calB}, \vphantom{M'}_\calA M'_{\calB}),\bdy)$ into a chain complex.
The $\calA_{\infty}$ analogue of chain maps are morphisms $f$ such that $\bdy(f) = 0$,
and we call these {\em $\calA_{\infty}$-homomorphisms}.
We say that morphisms $f$ and $f'$ are 
{\em $\calA_{\infty}$-homotopic} if $f -f' \in \im(\bdy)$.
Given another morphism $g: \vphantom{M'}_\calA M' _{\calB} \rightarrow \vphantom{M''}_{\calA}M''_{\calB}$, define
\begin{align}\label{A infinity homomorphism composition}
\xymatrix
{
&&\ar@{=>}[d]&\ar@{-->}[dd]&\ar@{=>}[d]&\\
&&\Delta_2\ar@/_/@{=>}[dr]\ar@/_/@{=>}[ddr]&&\Delta_2\ar@/^/@{=>}[dl]\ar@/^/@{=>}[ddl]&\\
g\circ f&=&&f\ar@{-->}[d]&&\\
&&&g\ar@{-->}[d]&&\\
&&&&.&
} 
\end{align}
We say that an $\calA_{\infty}$-homomorphism $f$ is an {\em $\calA_{\infty}$-homotopy equivalence} if 
there exists some $\calA_{\infty}$-homomorphism $g$ such that $g\circ f$ is $\calA_{\infty}$-homotopic to $\mathbb{I}_{\calM}$
and $f \circ g$ is $\calA_{\infty}$-homotopic to $\mathbb{I}_{\calM'}$.
Here $\mathbb{I}_{\vphantom{M}_\calA M_{\calB}}$ is the morphism $\vphantom{M}_\calA M_{\calB} \rightarrow \vphantom{M}_\calA M_{\calB}$ with
$\mathbb{I}_{\vphantom{M}_\calA M_{\calB}}(a_1\otimes...\otimes a_i\otimes x \otimes b_1 \otimes ... \otimes b_j)$ defined to be $x$ if $i = j = 0$
and zero otherwise (and $\mathbb{I}_{\vphantom{M}_\calA M'_{\calB}}$ is defined similarly).

\begin{theorem}\label{isotopyinvariance}
Let $\phi_b$ and $\phi_{nb}$ be diffeomorphisms representing the same element of $\MCG_0(F)$. Then 
$M(\phi_b)$ and $M(\phi_{nb})$ are homotopy equivalent as $\calA_{\infty}$-bimodules.
\end{theorem}
The main ingredient of the proof will be Lemma 8.6 of [\ref{LOT10b}], suitably modifed for $\calA_{\infty}$-bimodules:
\begin{lemma}\label{hptlemma}
Let $\vphantom{M}_{\calA}M_{\calB}$ be an $\calA_{\infty}$-bimodule over $\calA_{\infty}$-algebras $\calA,\calB$ over $\kk$,
let $M$ denote its underlying chain complex over $\kk$, and let $\tf_1: N \rightarrow M$ be a homotopy equivalence of chain complexes.
Then we can find
\begin{itemize}
 \item an $\calA_{\infty}$-bimodule structure $\vphantom{\calN}_{\calA}N_{\calB}$ on $N$ and
 \item an $\calA_{\infty}$-homotopy equivalence $F: \vphantom{N}_{\calA}N_{\calB} \rightarrow \vphantom{M}_{\calA}M_{\calB}$ with the property 
that $F_1 = \tf_1$.
\end{itemize}
Moreover, the structure map $\ovl{m}$ of $\vphantom{N}_{\calA}N_{\calB}$ is given as follows.
Let $\tg_1: M \rightarrow N$ be a homotopy inverse to $\tf_1$ and let $T: M \rightarrow M$ be a homotopy between $\tf_1 \circ \tg_1$ and $\mathbb{I}_M$.
For $a_1,...,a_m \in A,\;b_1,...,b_n \in B$, and $r \geq 1$, let 
\begin{align*}
&\mu_r^*\left((a_1\otimes...\otimes a_m)\otimes(b_1\otimes...\otimes b_n)\right) := \\
&\sum \left((a_{i_0+1}\otimes ... \otimes a_{i_1})\otimes (a_{i_1+1}\otimes ... \otimes a_{i_2}) \otimes ... \otimes (a_{i_{r-1}+1}\otimes...\otimes a_{i_r})\right) \\
&\;\;\;\;\otimes \left((b_{j_0+1}\otimes ... \otimes b_{j_1})\otimes (b_{j_1+1}\otimes ... \otimes b_{j_2}) \otimes ... \otimes (b_{j_{r-1}+1}\otimes...\otimes b_{j_r})\right), 
\end{align*}
where the sum is over all $0 = i_0 \leq i_1 \leq ... \leq i_r = m,\;0 = j_0 \leq j_1 \leq ... \leq j_r = n$ such that 
$(i_k,j_k) \neq (i_{k-1},j_{k-1})$ for $1 \leq k \leq r$.
Let $a = (a_1 \otimes ... \otimes a_m)$ and $b = (b_1 \otimes ... \otimes b_n)$.
Then $\ovl{m}$ is given by
\begin{align}\label{hptmaps}
\xymatrix@=1em
{
\ovl{m}(a \otimes y \otimes b) & = & a\ar@/_/@{=>}[dr] && b\ar@/^/@{=>}[dl] &+& a\ar@/_/@{=>}[dr] && b\ar@/^/@{=>}[dl] &+& a\ar@/_/@{=>}[dr] && b\ar@/^/@{=>}[dl] &+\;...\\
&&& \mu^*_1\ar@/^1pc/@{=>}[ddd]\ar@/_1pc/@{=>}[ddd] &&&& \mu^*_2\ar@/^1pc/@{=>}[ddd]\ar@/_1pc/@{=>}[ddd]\ar@/^2pc/@{=>}[ddddd]\ar@/_2pc/@{=>}[ddddd] &&&& \mu^*_3\ar@/^1pc/@{=>}[ddd]\ar@/_1pc/@{=>}[ddd]\ar@/^2pc/@{=>}[ddddd]\ar@/_2pc/@{=>}[ddddd]\ar@/^3pc/@{=>}[ddddddd]\ar@/_3pc/@{=>}[ddddddd]\\
&&& y\ar@{-->}[d] &&&& y\ar@{-->}[d] &&&& y\ar@{-->}[d] && \\
&&& \tf_1\ar[d] &&&& \tf_1\ar[d] &&&& \tf_1\ar[d]\\
&&& m\ar[d] &&&& m\ar[d] &&&& m\ar[d]\\
&&& \tg_1\ar@{-->}[d] &&&& T\ar[d] &&&& T\ar[d]\\
&&&   &&&& m\ar[d] &&&& m\ar[d]\\
&&&   &&&& \tg_1\ar@{-->}[d] &&&& T\ar[d]\\
&&&   &&&&  &&&& m\ar[d]\\
&&&   &&&&  &&&& \tg_1\ar@{-->}[d]\\
&&&&&&&&&&&
} 
\end{align}
\end{lemma}

By a standard result of curves on surfaces (see for example [\ref{MCG}]), $\calD(\phi_b)$ and $\calD(\phi_{nb})$ differ by a sequence
of finger moves, i.e. isotopies adding or removing an innermost bigon. Then it suffices to consider the case that $\calD(\phi_{nb})$ is obtained from $\calD(\phi_b)$ 
by removing an innermost bigon $B$. 
In particular we assume that an isotopy between $\calD(\phi_b)$ and $\calD(\phi_{nb})$ is the identity outside of a small neighborhood of $B$, and that the local picture inside the neighborhood
is as in Figure \ref{FingerMove}.
It will also be convenient to further assume that the symmetric difference between the $\beta$ curves of $\calD(\phi_b)$ and those of $\calD(\phi_{nb})$
has exactly two components, i.e. there are no superfluous bigons created between the $\beta$ curves of the two diagrams.

\begin{figure}
\begin{center}
 \includegraphics[scale=.6]{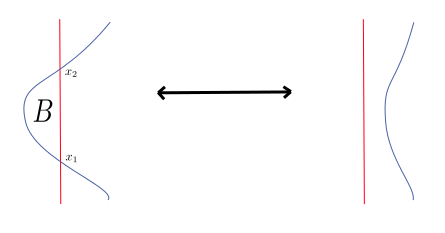}
\end{center}
\caption{A ``finger move'' showing how $\calD(\phi_b)$ and $\calD(\phi_{nb})$ are related.}
\label{FingerMove}
\end{figure}

Let $\{x_1,...,x_n\}$ denote the canonical basis for $M(\phi_b)$, where the bigon $B$ to be removed is from $x_1$ to $x_2$.
We will write $x_i \rightarrow x_j$ if $x_j$ is a summand of $m_1(x_i)$ (in the basis implied by the context).
We begin by introducing a new basis to isolate $B$. Specifically, let
\begin{align*}
\tilde{x}_1 &:= x_1\\
\tilde{x}_2 &:= m_1(x_1)\\
\tilde{x}_i &:= x_i + x_1\;\;\text{whenever}\;x_i \rightarrow x_2\;\;(\text{for}\;3 \leq i \leq n)\\
\tilde{x}_i &:= x_i\;\;\text{otherwise}
\end{align*}

The reader can easily verify that $\{\tilde{x}_1,...,\tilde{x}_n\}$ again forms a basis,
and that in this basis the only differential relation involving $\tx_1$ or $\tx_2$ is $m_1(\tx_1) = \tx_2$.
That is, representing the chain complex of $M(\phi_b)$ as a directed graph as in Figure \ref{GraphChangeBasis},
there are no arrows entering or leaving $\tx_1$ or $\tx_2$ except for a single arrow $\tx_1 \rightarrow \tx_2$.

\begin{figure}
\begin{center}
 \includegraphics[scale=.6]{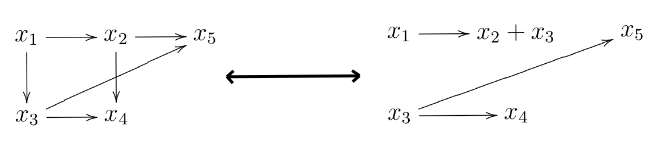}
\end{center}
\caption{Changing basis in a chain complex to isolate a bigon. The arrows here represent the differential.}
\label{GraphChangeBasis}
\end{figure}

Now let $\{y_3,...,y_n\}$ denote the canonical basis for $M(\phi_{nb})$, and observe that we have
a natural correspondence $y_i \leftrightarrow x_i$ between the generators for $i \geq 3$.
We define maps $f_1: M(\phi_{nb}) \rightarrow M(\phi_b)$ and $g_1: M(\phi_b) \rightarrow M(\phi_{nb})$ by
\begin{align*}
f_1(y_i) := x_i \;\; 3\leq i \leq n\;\;\;\;\;\;\;\;\;\;\;\;\;\;\; &g_1(x_i) := y_i\;\;3 \leq i \leq n\\
& g_1(x_i) := 0\;\; i = 1,2.
\end{align*}
We also define maps $\tf_1: M(\phi_{nb}) \rightarrow M(\phi_b)$ and $\tg_1: M(\phi_b) \rightarrow M(\phi_{nb})$ by
\begin{align*}
\tf_1(y_i) := \tx_i \;\; 3\leq i \leq n\;\;\;\;\;\;\;\;\;\;\;\;\;\;\; &\tg_1(\tx_i) := y_i\;\;3 \leq i \leq n\\
& \tg_1(\tx_i) := 0\;\; i = 1,2.
\end{align*}

\begin{lemma}\label{inversehomotopyequivalences}
The maps $\tf_1$ and $\tg_1$ are inverse homotopy equivalences of the underlying chain complexes of $M(\phi_{nb})$ and $M(\phi_b)$. 
\end{lemma}

Before proving this lemma, it will be useful to give a general result relating polygons in $\calD(\phi_{nb})$ to polygons in $\calD(\phi_b)$.
Let $S_1^{(\sigma_1,...,\sigma_m),(\rho_1,...,\rho_n)}(x_i,x_j)$ denote the set of polygons in $\calD(\phi_b)$ from $x_i$ to $x_j$ through
$(\sigma_1,...,\sigma_m)$ and $(\rho_1,...,\rho_n)$.
For $r \geq 2$, let $S_r^{(\sigma_1,...,\sigma_m),(\rho_1,...,\rho_n)}(x_i,x_j)$ denote the set of tuples $(P_1,...,P_r)$ of polygons in $\calD(\phi_b)$ such that
\begin{itemize}
 \item $P_1$ is a polygon, distinct from $B$, from $x_i$ to $x_2$ through $(\sigma_{s_0+1},...,\sigma_{s_1})$ and $(\rho_{t_0+1},...,\rho_{t_1})$,

 \item $P_k$ is a polygon, distinct from $B$, from $x_1$ to $x_2$ through $(\sigma_{s_{k-1}+1},...,\sigma_{s_k})$ and $(\rho_{t_{k-1}+1},...,\rho_{t_k})$ for $2 \leq k \leq r-1$, and 

 \item $P_r$ is a polygon, distinct from $B$, from $x_1$ to $x_j$ through $(\sigma_{s_{r-1}+1},...,\sigma_{s_r})$ and $(\rho_{t_{r-1}+1},...,\rho_{t_r})$

\end{itemize}
for some $0 = s_0 \leq s_1 \leq ... \leq s_r = m$ and $0 = t_0 \leq t_1 \leq ... \leq t_r = n$.

We will also set
\begin{align*}
S^{(\sigma_1,...,\sigma_m),(\rho_1,...,\rho_n)}(x_i,x_j) := \cup_{r \geq 1} S_r^{(\sigma_1,...,\sigma_m),(\rho_1,...,\rho_n)}(x_i,x_j).
\end{align*}

\begin{lemma}\label{relatepolygons} 
There is a natural bijective correspondence between the elements of $S^{(\sigma_1,...,\sigma_m),(\rho_1,...,\rho_n)}(x_i,x_j)$
and polygons in $\calD(\phi_{nb})$ from $y_i$ to $y_j$ through $(\sigma_1,...,\sigma_m)$ and $(\rho_1,...,\rho_n)$ (for $i,j \geq 3$).
\end{lemma}
\begin{proof}
For a polygon $P: \D^2 \rightarrow \calD(\phi_{nb})$ from $y_i$ to $y_j$ through $(\sigma_1,...,\sigma_m)$ and $(\rho_1,...,\rho_n)$,
let $P': \D^2 \rightarrow \calD(\phi_b)$ be the corresponding map into $\calD(\phi_b)$.
Observe that the components of $(P')^{-1}(\boldsymbol{\beta})$ intersecting the left side of $\Dom(P')$
cut $\Dom(P')$ into various connected domains
(here we are implicitly relying on the symmetric difference assumption above to avoid unnecessary ``ripples'' between the $\beta$ curves).
Let $D_1,...,D_r$ denote those domains whose 
boundaries map entirely into $\boldsymbol{\alpha} \cup \boldsymbol{\beta} \cup \bdy \calD(\phi_b)$. Then the restrictions
$(P'|_{D_1},...,P'|_{D_r})$ define an element of $S_r^{(\sigma_1,...,\sigma_m),(\rho_1,...,\rho_n)}(x_i,x_j)$.

Conversely, suppose $(P_1,...,P_r) \in S_r^{(\sigma_1,...,\sigma_m),(\rho_1,...,\rho_n)}(x_i,x_j)$.
Superimposing $\calD(\phi_b)$ and $\calD(\phi_{nb})$, there is a rectangular region $R$ which has three $\beta$ sides and one $\alpha$ side,
where the $\alpha$ side is also a side of $B$.
There is a natural way to glue a copy of $R$ along one of its sides to a segment of the left side of $P_k$,
and along another side to a segment of the left side of $P_{k+1}$, for $1 \leq k < r$, to obtain
a polygon 
\begin{align*}
P: \Dom(P_1) \cup R \cup \Dom(P_2) \cup ... \cup R \cup \Dom(P_r): \rightarrow \calD(\phi_{nb})
\end{align*}
 from $g_1(\IP(P_1))$ to $g_1(\TP(P_r))$ through $(\sigma_1,...,\sigma_m)$ and $(\rho_1,...,\rho_n)$.

The reader can easily check that these operations are inverses.
\end{proof}

Suppose for each $i$ we have 
\begin{align*}
m_1(x_i) = a_1^ix_1 + a_2^ix_2 + ... + a_n^ix_n,
\end{align*}
where $a_k^i \in \F_2$ for $1 \leq k \leq n$.

\begin{corollary}\label{relatebigons}
For $i \geq 3$, we have
\begin{align*}
m_1(y_i) = g_1(m_1(x_i)) + a_2^ig_1(m_1(x_1)). 
\end{align*}
\end{corollary}
\begin{proof}
This follows from the fact that $S_r^{(),()}(x_i,x_j)$ is the empty set for $r \geq 3$ since $B$ is the unique bigon from $x_1$ to $x_2$.
The first term comes from $S_1^{(),()}(x_i,x_j)$ and the second term comes from $S_2^{(),()}(x_i,x_j)$.
\end{proof}

\begin{proof}[Proof of Lemma \ref{inversehomotopyequivalences}]
We begin by showing that $\tf_1$ and $\tg_1$ are chain maps, i.e. that they commute with the $m_1$ maps of $M(\phi_{nb})$ and $M(\phi_b)$.
Since $m_1(m_1(x_i)) = 0$, in particular $x_2$ is not a summand of $m_1(m_1(x_i))$, and therefore we can write
\begin{align*}
m_1(x_i) &= a_1^i(x_1 + x_1) + a_2^ix_2 + a_3^i\tx_3 + a_4^i\tx_4 + ... + a_n^i\tx_n\\
&= a_2^ix_2 + a_3^i\tx_3 + a_4^i\tx_4 + ... + a_n^i\tx_n,
\end{align*}
since we have only added an even number of $x_1$'s.
Then we have
\begin{align*}
m_1(\tx_i) = m_1(x_i + a_2^ix_1) = a_2^i(x_2 + m_1(x_1)) + a_3^i\tx_3 + a_4^i\tx_4 + ... + a_n^i\tx_n
\end{align*}
and therefore
\begin{align*}
\tg_1(m_1(\tx_i)) &= a_2^i\tg_1(x_2 + m_1(x_1)) + a_3^iy_3 + a_4^iy_4 + ... + a_n^iy_n\\
&= a_2^i\tg_1(a_1^1x_1 + a_3^1x_3 + a_4^1x_4 + ... + a_n^1x_n) + g_1(m_1(x_i))\\
&= a_2^i\tg_1(a_1^1[x_1 + x_1] + a_3^1\tx_3 + a_4^1\tx_4 + ... + a_n^1\tx_n) + g_1(m_1(x_i))\\
&= a_2^i(a_3^1y_3 + a_4^1y_4 + ... + a_n^1y_n) + g_1(m_1(x_i))\\
&= a_2^ig_1(m_1(x_1)) + g_1(m_1(x_i))\\
&= m_1(\tg_1(\tx_i)),
\end{align*}
where the third line follows from another evenness argument and the last line follows from Corollary \ref{relatebigons}.
Then evidently $\tg_1\circ m_1 = m_1 \circ \tg_1$, and it easily follows from the definitions of $\tf_1$ and $\tg_1$ that $\tf_1 \circ m_1 = m_1 \circ \tf_1$. 

Since $\tg_1 \circ \tf_1$ is the identity, to complete the proof it suffices to find a homotopy $T: M(\phi_b) \rightarrow M(\phi_b)$ from 
$\tf_1 \circ \tg_1$ to $\mathbb{I}$, i.e. a map $T$ such that $\tf_1\circ \tg_1 - \mathbb{I} = T\circ m_1 + m_1\circ T$.
Since $\tf_1\circ \tg_1 - \mathbb{I}$ fixes $\tx_1$ and $\tx_2$ and sends every other $\tx_i$ to zero, one easily checks that $T$ can be chosen 
to be the map sending $\tx_2$ to $\tx_1$ and sending $\tx_1,\tx_3,...,\tx_n$ to zero.
\end{proof}
  
We are now finally in a position to prove Theorem \ref{isotopyinvariance}.
To aid the proof, let $\#_{x_k}: M(\phi_b) \rightarrow \F_2$ and $\#_{y_k}: M(\phi_{nb}) \rightarrow \F_2$ be functions counting the number 
of $x_k$ and $y_k$ summands respectively of the expansion of the input in the corresponding basis.

\begin{proof}[Proof of Theorem \ref{isotopyinvariance}]
By Lemma \ref{hptlemma}, $M(\phi_b)$ is $\calA_{\infty}$-homotopy equivalent to an induced $\calA_{\infty}$ structure $\ovl{m}$ on the underlying 
chain complex of $M(\phi_{nb})$, and therefore it will suffice to show that the map $\ovl{m}$ is the same as the usual structure map $m$ for $M(\phi_{nb})$.
In fact, since $\ovl{m}$ is given by (\ref{hptmaps}), where $T$ is the map defined at the end of the proof of Lemma \ref{inversehomotopyequivalences}
sending $\tx_2$ to $\tx_1$ and $\tx_1,\tx_3,...,\tx_n$ to zero,
unwinding the diagram shows that that we have
\begin{align}\label{initialmbarexpansion}
&\ovl{m}(\sigma_1,...,\sigma_m,y_i,\rho_1,...,\rho_n) = \sum_{j=3}^n y_j \;\#_{y_j}(\tg_1(m(\sigma_1,...,\sigma_m,\tx_i,\rho_1,...,\rho_n)))\\
&+\sum_{j=3}^n y_j \sum \#_{\tx_2}(m(\sigma_{s_0+1},...,\sigma_{s_1},\tx_i,\rho_{t_0+1},...,\rho_{t_1}))
\#_{\tx_2}(m(\sigma_{s_1+1},...,\sigma_{s_2},\tx_1,\rho_{t_1+1},...,\rho_{t_2}))...\nn\\
&\#_{\tx_2}(m(\sigma_{s_{r-2}+1},...,\sigma_{s_{r-1}},\tx_1,\rho_{t_{r-2}+1},...,\rho_{t_{r-1}}))
\#_{y_j}(\tg_1(m(\sigma_{s_{r-1}+1},...,\sigma_{s_r},\tx_1,\rho_{t_{r-1}+1},...,\rho_{t_r})))\nn,
\end{align}

where the last sum is over all $r \geq 2$ and all $0 = s_0 \leq s_1 \leq ... \leq s_r = m,\; 0 = t_0 \leq t_1 \leq ... \leq t_r = n$
such that $(s_k,t_k) \neq (s_{k-1},t_{k-1})$ for $1 \leq k \leq r$.

\begin{lemma}\label{countingequation}
For any $i,j \geq 3,\;\sigma = (\sigma_1,...,\sigma_s)$, and $\rho = (\rho_1,...,\rho_t)$, we have
\begin{align*}
&\#_{y_j}(\tg_1(m(\sigma,\tx_i,\rho))) =\\&\;\;\;\;\; \#_{x_j}(m(\sigma,x_i,\rho)) \;+ \#_{x_j}(m(\sigma,x_1,\rho))\#_{x_2}(m(x_i)) \;+ \\
&\;\;\;\;\;\#_{x_j}(m(x_1))\#_{x_2}(m(\sigma,x_i,\rho)) \;+ 
\#_{x_j}(m(x_1))\#_{x_2}(m(\sigma,x_1,\rho))\#_{x_2}(m(x_i)).
\end{align*}
 
\end{lemma}
\begin{proof}
From the definitions of $\tg_1$ and the basis $\{\tx_k\}$ (and in particular the fact that, for $j \leq 3$, $\tx_j$ appears only in the expansion of $x_j$ and 
possibly in the expansion of $x_2$ if $x_1 \rightarrow x_j$), we have
\begin{align*}
\#_{y_j}(\tg_1(m(\sigma,\tx_i,\rho))) &= \#_{\tx_j}(m(\sigma,\tx_1,\rho))\\
&= \#_{x_j}(m(\sigma,\tx_i,\rho)) \;+ \#_{x_j}(m(x_1))\#_{x_2}(m(\sigma,\tx_i,\rho)), 
\end{align*}
as well as 
\begin{align*}
\tx_i = x_i + \#_{x_2}(m(x_i))x_1,
\end{align*}
and these combine to give the desired result.
\end{proof}
With a little work, we can now simplify (\ref{initialmbarexpansion}) using the result from Lemma \ref{countingequation}, as well as the facts
$\tx_1 = x_1$ and $\#_{\tx_2} = \#_{x_2}$, to obtain
\begin{align}\label{improvedmbarexpansion}
&\ovl{m}(\sigma_1,...,\sigma_m,y_i,\rho_1,...,\rho_n) = \sum_{j=3}^n y_j \;\#_{x_j}(m(\sigma_1,...,\sigma_m,x_i,\rho_1,...,\rho_n))\\
&+\sum_{j=3}^n y_j \sum \#_{x_2}(m(\sigma_{s_0+1},...,\sigma_{s_1},x_i,\rho_{t_0+1},...,\rho_{t_1}))
\#_{x_2}(m(\sigma_{s_1+1},...,\sigma_{s_2},x_1,\rho_{t_1+1},...,\rho_{t_2}))...\nn\\
&\#_{x_2}(m(\sigma_{s_{r-2}+1},...,\sigma_{s_{r-1}},x_1,\rho_{t_{r-2}+1},...,\rho_{t_{r-1}}))
\#_{x_j}(m(\sigma_{s_{r-1}+1},...,\sigma_{s_r},x_1,\rho_{t_{r-1}+1},...,\rho_{t_r}))\nn,
\end{align}

where the last sum is over all $r \geq 2$ and all $0 = s_0 \leq s_1 \leq ... \leq s_r = m,\; 0 = t_0 \leq t_1 \leq ... \leq t_r = n$
such that $(s_k,t_k) \neq (s_{k-1},t_{k-1})$ for $2 \leq k \leq r-1$ (note the difference in indexing from (\ref{initialmbarexpansion})).
But Lemma \ref{relatepolygons} shows that this last sum also computes $m$
(note that the condition $(s_k,t_k) \neq (s_{k-1},t_{k-1})$ for $2 \leq k \leq r-1$ corresponds to the stipulation in the definition of 
$S_r^{(\sigma_1,...,\sigma_m),(\rho_1,...,\rho_n)}(x_i,x_j)$ that the polygons be distinct from $B$), and this completes the proof.
\end{proof}

\section{The Main Result}\label{The Main Result}
In Section \ref{Various types of bimodules} we define type DD bimodules and DA bimodules, and in particular we define the type DD bimodule $\dd$, which plays a key role.
We also define the box product $\boxtimes$, a kind of tensor product which can be performed on the various types of bimodules.
In Section \ref{The gluing construction} we show how to obtain a representation of $\calD(\psi_0 \circ \phi_0)$ by gluing together $\calD(\phi_0),\calD(\mathbb{I})$, and $\calD(\psi_0)$.
In Section \ref{Checking the composition behavior} we prove the crucial ingredient for our mapping class group action, which is that
$M(\psi \circ \phi)$ and $M(\phi) \boxtimes (\dd \boxtimes M(\psi))$ are $\calA_{\infty}$-homotopy equivalent.
Finally, in Section \ref{Completing the proof} we complete the proof that we have a faithful action.

\subsection{Various types of bimodules}\label{Various types of bimodules}

We have already defined type AA bimodules, which are just $\calA_{\infty}$-bimodules. We now define type DD and type DA bimodules.
\begin{definition}
Let $\calA = (A,\bdy_A)$ and $\calB = (B,\bdy_B)$ be differential algebras over $\kk$ (i.e. $\bdy_A$ and $\bdy_B$ are $\kk$-linear and satisfy the Leibniz rule)
such that $\bdy_A\circ \bdy_A = 0$ and $\bdy_B \circ \bdy_B = 0$. 
A {\em type DD} bimodule $^{\calA}M^{\calB}$ over $\calA$ and $\calB$ is a $\kk$-bimodule $M$ and a map $\delta^1_{DD}: M \rightarrow A \otimes M \otimes B$ such that the following compatibility condition holds:
\begin{align*}
\xymatrix@=1.25em
{
&\ar@{-->}[d]&&&&\ar@{-->}[d]&&&&\ar@{-->}[d]&&\\
  &\delta^1_{DD}\ar@/_/[ddl]\ar@{-->}[d]\ar@/^/[ddr]&&+&&\delta^1_{DD}\ar@{-->}[dd]\ar@/_/[dl]\ar@/^/[ddr]&&+&&\delta^1_{DD}\ar@{-->}[dd]\ar@/_/[ddl]\ar@/^/[dr]&&=&0,\\
  &\delta^1_{DD}\ar@/_/[dl]\ar@/^/[dr]\ar@{-->}[dd]&&&\bdy_A\ar[d]&&&&&&\bdy_B\ar[d]&\\
  \mu_2\ar[d]&&\mu_2\ar[d]&&&&&&&&&\\
&&&&&&&&&&&\\
} 
\end{align*}
where $\mu_2$ denotes algebra multiplication.

For later purposes we also define an auxilary map $\delta_{DD}: M \rightarrow T^*(A) \otimes M \otimes T^*(B)$ by
\begin{align*}
\xymatrix@=1em
{
&\ar@{-->}[d]&&&\ar@{-->}[ddd]&&&\ar@{-->}[d]&&&&\ar@{-->}[d]&&&&\\
&\delta_{DD}\ar@/_/@{=>}[ddl]\ar@{-->}[dd]\ar@/^/@{=>}[ddr]&&:=&&+&&\delta^1_{DD}\ar@/_/[ddl]\ar@{-->}[dd]\ar@/^/[ddr]&&+&&\delta^1_{DD}\ar@/_/[ddl]\ar@{-->}[d]\ar@/^/[ddr]&&+&...&\\
&&&&&&&&&&&\delta^1_{DD}\ar@/_/[dl]\ar@{-->}[d]\ar@/^/[dr]&&&&\\
&&&&&&&&&&&&&&&\\
} 
\end{align*}

\end{definition}
Let $\delta^0_{DD} = \mathbb{I}$, let $\delta^2_{DD} := (\mathbb{I}\otimes\delta^1_{DD}\otimes\mathbb{I}) \circ \delta^1_{DD}$,
and more generally define $\delta^n_{DD}$ inductively by $\delta^n_{DD} := (\mathbb{I}\otimes\delta^{n-1}_{DD}\otimes\mathbb{I}) \circ \delta^1_{DD}$.
The above equation can then also be written as
\begin{align*}
\delta_{DD} = \sum_{i\geq 0}\delta_{DD}^i
\end{align*}

Type DA bimodules, which we define presently, are a sort of hybrid of type DD and type AA bimodules.
\begin{definition}
Let $\calA$ and $\calB$ be $\calA_{\infty}$-algebras over $\kk$. A {\em type DA} bimodule $^{\calA}M_{\calB}$ over $\calA$ and $\calB$ consists of a $\kk$-bimodule $M$
and $\kk$-linear maps 
\begin{align*}
\delta^{1,1+j}_{DA}: M \otimes B^{\otimes j} \rightarrow A \otimes M 
\end{align*}
satisfying a compatibility condition given as follows. 
Let $\delta^1_{DA}: M \otimes T^*(B) \rightarrow A\otimes M$ be the total map given by $\delta^1_{DA} = \sum_j \delta^{1,j}_{DA}$.
Define a map $\delta_{DA}: M \otimes T^*(B) \rightarrow T^*(A)\otimes M$ by
\begin{align*}
\xymatrix@C=.9em@R=2.2em
{
&&&&&&&&&&&\ar@{-->}[dd]&\ar@{=>}[d]&&&\ar@{-->}[dd]&\ar@{=>}[d]&\\  
&\ar@{-->}[d]&\ar@/^/@{=>}[dl]&&\ar@{-->}[dd]&&&\ar@{-->}[d]&\ar@/^/@{=>}[dl]&&&&\Delta_2\ar@/^/@{=>}[dl]\ar@/^/@{=>}[ddl]&&&&\Delta_3\ar@/^/@{=>}[dl]\ar@/^/@{=>}[ddl]\ar@/^/@{=>}[dddl]&\\
&\delta_{DA}\ar@/_/@{=>}[dl]\ar@{-->}[d]&&:=&&\oplus&&\delta^1_{DA}\ar@/_/[dl]\ar@{-->}[d]&&\oplus&&\delta^1_{DA}\ar@/_/[ddl]\ar@{-->}[d]&&\oplus&&\delta^1_{DA}\ar@/_/[dddl]\ar@{-->}[d]&&\oplus\; ...\\
&&&&&&&&&&&\delta^1_{DA}\ar@/_/[dl]\ar@{-->}[d]&&&&\delta^1_{DA}\ar@/_/[ddl]\ar@{-->}[d]&&\\
&&&&&&&&&&&&&&&\delta^1_{DA}\ar@/_/[dl]\ar@{-->}[d]&&\\
&&&&&&&&&&&&&&&&&\\
} 
\end{align*}
Here $\Delta_3 := (\Delta_2 \otimes \mathbb{I}) \circ \Delta_2$ (see Section \ref{algebras and bimodules}),
and in general $\Delta_n$ is defined inductively by $\Delta_n := (\Delta_{n-1} \otimes \mathbb{I}) \circ \Delta_2$.

The compatibility condition is then given by
\begin{align*}
\xymatrix
{
&\ar@{-->}[dd]&\ar@{=>}[d]&&&\ar@{-->}[d]&\ar@/^/@{=>}[dl]&&\\
&&\ovl{D}^{\calB}\ar@/^/@{=>}[dl]&&&\delta_{DA}\ar@/_/@{=>}[dl]\ar@{-->}[dd]&&\\
&\delta_{DA}\ar@/_/@{=>}[dl]\ar@{-->}[d]&&+&\ovl{D}^{\calA}\ar@{=>}[d]&&&=\;0.\\
&&&&&&&
} 
\end{align*}
\end{definition}

\begin{remark}
One can also define type AD bimodules, which are essentially the mirror images of type DA bimodules. Since we will not need these here, we omit them. 
\end{remark}

With these definitions at hand we now define the type DD bimodule $\dd$.
Call a chord in $\calB(\calZ)$ {\em short} if it connects adjacent points, and let $SC(\calZ)$ denote the set of short chords in $\calB(\calZ)$.
For a short chord $\xi$ in $\calB(\calZ)$, there is a corresponding short chord $\xi'$ in $\calB(\calZ')$ defined such that
$\xi$ and $\xi'$ lie on the boundary of the same connected component of $\calD(\mathbb{I})$.
\begin{definition}
For each generator $x_i \in M(\mathbb{I})$, let $I(x_i)$ denote the corresponding idempotent in $\calB(\calZ)$ and let $J(x_i)$ denote the corresponding
idempotent in $\calB(\calZ')$. 
Let
\begin{align*}
\dd := \bigoplus_i \calB(\calZ)I(x_i)\otimes J(x_i)\calB(\calZ'). 
\end{align*}
We will abuse notation and let $x_i$ denote a generator of the summand corresponding to $x_i$. We define a differential by
\begin{align*}
\bdy(x_i) = \sum_j\sum_{\xi\in SC(\calZ)}I(x_i)\cdot \xi\cdot x_j\cdot \xi'\cdot J(x_i), 
\end{align*}
and we extend this to all of $\dd$ by the Leibniz rule.
\end{definition}

Next we define $\boxtimes$ for bimodules. Specifically,
\begin{itemize}
 \item $DD \boxtimes AA$ is a type $DA$ bimodule.
 \item $AA \boxtimes DA$ is a type $AA$ bimodule.
\end{itemize}

\begin{remark}
We can of course define $\boxtimes$ for other combinations of bimodules as well, although we will not need them here. 
\end{remark}

\begin{definition}
Let $\vphantom{M}^{\calA}M^{\calB}$ be a $DD$ bimodule and let $\vphantom{N}_{\calB}N_{\calC}$ be an $AA$ bimodule. 
Then $(\vphantom{M}^{\calA}M^{\calB})\boxtimes(\vphantom{N}_{\calB}N_{\calC})$ as a $\kk$-bimodule
is $M \otimes_\kk N$,
with type DA structure map $\delta^1_{DA}$ given by 
\end{definition}
\begin{align*}
\xymatrix
{
&\ar@{-->}[d]&\ar@{-->}[dd]&\ar@/^/@{=>}[ddl]\\
&\delta_{DD}\ar@/_/@{=>}[dl]\ar@/^/@{=>}[dr]\ar@{-->}[dd]&&\\
\Pi\ar[d]&&m\ar@{-->}[d]&\\
&&&,\\
}
\end{align*}
where $\Pi$ denotes the algebra multiplication in $\calA$.

\begin{definition}
Let $\vphantom{M}_{\calA}M_{\calB}$ be an $AA$ bimodule and let $\vphantom{N}^{\calB}N_{\calC}$ be a $DA$ bimodule. 
Then $(\vphantom{M}_{\calA}M_{\calB}) \boxtimes (\vphantom{N}^{\calB}N_{\calC})$ as a $\kk$-bimodule
is $M \otimes_\kk N$,
with type AA structure map given by 
\end{definition}
\begin{align*}
\xymatrix
{
\ar@/_/@{=>}[ddr]&\ar@{-->}[dd]&\ar@{-->}[d]&\ar@/^/@{=>}[dl]\\
&&\delta_{DA}\ar@/^/@{=>}[dl]\ar@{-->}[dd]&\\
&m\ar@{-->}[d]&&\\
&&&.\\
} 
\end{align*}

\begin{remark}
The product $\boxtimes$ enjoys a number a desirable properties for tensor products. The reader is encouraged to consult Section 2.3 of [\ref{Bimodules}] for a detailed treatment.
We also mention here that $\boxtimes$ is not strictly associative, which is why we must choose a parenthesization in Theorem \ref{bigtheorem}.
\end{remark}

\subsection{The gluing construction for $\calD(\psi_0\circ\phi_0)$}\label{The gluing construction}
Since our goal is to relate $M(\psi \circ \phi)$ to $M(\phi)$ and $M(\psi)$, we will need a way to relate polygons in $\calD(\psi_0\circ\phi_0)$ to polygons
in $\calD(\phi_0)$ and $\calD(\psi_0)$. We will think of $\calD(\phi_0)$ (or $\calD(\mathbb{I})$ or $\calD(\psi_0)$)
as a space with distinguished $\alpha$ and $\beta$ curves on it, so that for example cutting or gluing $\calD(\phi_0)$ means cutting or gluing the $\alpha$ and $\beta$ arcs as well.

Recall that for a fixed arc diagram $\calZ$ on a surface $F$, $F^{\circ}(\calZ)$ is a surface 
with the same genus as $F$ and twice as many boundary components.
We can view $\bdy(F^{\circ}(\calZ))$ as the union of $\bdy F(\calZ)$ and $\bdy F(\calZ')$, and we call these the $\rho$ and $\sigma$ boundaries respectively.
Let $\calD(\phi_0) \cup_{\sigma} -\calD(\mathbb{I}) \cup_{\rho} \calD(\psi_0)$ be obtained from 
the disjoint union $\calD(\phi_0)\amalg -\calD(\mathbb{I})\amalg \calD(\psi_0)$
by identifying the $\rho$ boundary of $\calD(\phi_0)$ with the $\rho$ boundary of $-\calD(\mathbb{I})$ and by identifying the $\sigma$ boundary of 
$-\calD(\mathbb{I})$ with the $\sigma$ boundary of $\calD(\psi_0)$ (here $-\calD(\mathbb{I})$ denotes the orientation reversal of $\calD(\mathbb{I})$).
For convenience, we will assume throughout that there are no bigons in $\calD(\phi_0)$ or $\calD(\psi_0)$ (by Theorem \ref{isotopyinvariance} there is no loss of generality).

We will obtain a diagram $\tilde{\calD}(\psi_0 \circ \phi_0)$ by destabilizing $\calD(\phi_0) \cup_{\sigma} -\calD(\mathbb{I}) \cup_{\rho} \calD(\psi_0)$, a process which we now explain.
Notice that in $\calD(\phi_0) \cup_{\sigma} -\calD(\mathbb{I}) \cup_{\rho} \calD(\psi_0)$, the $\alpha$ arcs of $\calD(\phi_0)$ are glued to the $\alpha$ arcs of $-\calD(\mathbb{I})$ to form
{\em $\alpha$ circles}, and similarly the $\beta$ arcs of $-\calD(\mathbb{I})$ are glued to the $\beta$ arcs of $\calD(\psi_0)$ to form {\em$\beta$ circles}.
Each generator $w_j \in \dd$ corresponds to a unique intersection point of an $\alpha$ circle and a $\beta$ cirle in $\calD(\phi_0) \cup_{\sigma} -\calD(\mathbb{I}) \cup_{\rho} \calD(\psi_0)$.
Let $T_{w_j}$ denote a tubular neighborhood of this pair of intersecting $\alpha$ and $\beta$ circles in $\calD(\phi_0) \cup_{\sigma} -\calD(\mathbb{I}) \cup_{\rho} \calD(\psi_0)$,
where $T_{w_j}$ is sufficiently small that $\bdy T_{w_j} \cap \bdy (\calD(\phi_0) \cup_{\sigma} -\calD(\mathbb{I}) \cup_{\rho} \calD(\psi_0)) = \nil$
and such that further shrinking $T_{w_j}$ does not reduce the number of intersection points of $\bdy T_{w_j}$ with the 
$\alpha$ and $\beta$ arcs of $\calD(\phi_0),-\calD(\mathbb{I})$, and $\calD(\psi_0)$.
Notice that $T_{w_j}$ is a torus with one boundary component.

Let $\{x_i\},\{w_j\}$, and $\{y_k\}$ denote the canonical bases of $M(\phi_0),\dd$, and $M(\psi_0)$ respectively.
Let $\TT := \cup_j T_{w_j}$, and let us further suppose that the $T_{w_j}$'s are chosen to be pairwise disjoint.
\begin{construction}\label{Gluing Construction}
The construction of $\tilde{\calD}(\psi_0\circ\phi_0)$ involves 
removing each $T_{w_j}$ from $\calD(\phi_0) \cup_{\sigma} -\calD(\mathbb{I}) \cup_{\rho} \calD(\psi_0)$ and gluing in a corresponding disk $D_{w_j}$ along $\bdy T_{w_j}$.
That is, consider the space
\begin{align*}
\left((\calD(\phi_0) \cup_{\sigma} -\calD(\mathbb{I}) \cup_{\rho} \calD(\psi_0)) \setminus 
\TT\right) \cup_{\sim} \left(\amalg_j D_{w_j}\right),
\end{align*}
where $\sim$ identifies each $\bdy T_{w_j}$ diffeomorphically with $\bdy D_{w_j}$.
Let $\mathbf{D}$ denote $\cup_j D_{w_j}$.
Observe that each $x_i$ lies on a $\beta$ arc segment (i.e. connected subset) $t_{x_i} \subset T_{w_j}$, for some $T_{w_j}$,
 with $\bdy t_{x_i} \subset \bdy T_{w_j}$.
We additionally replace each $t_{x_i}$ with a $\beta$ arc segment $d_{x_i}$ in $D_{w_j}$, such that $\bdy t_{x_i}$ and $\bdy d_{x_i}$ are identified under $\sim$.
Similarly, each $y_k$  lies on an $\alpha$ arc segment $t_{y_k} \subset T_{w_j}$, for some $T_{w_j}$, with $\bdy t_{y_k} \subset \bdy T_{w_j}$.
We replace each $t_{y_k}$ with an $\alpha$ arc segment $d_{y_k}$ in $D_{w_j}$, such that $\bdy t_{y_k}$ and $\bdy d_{y_k}$ are identified under $\sim$.
The $d_{x_i}$'s and $d_{y_k}$'s are chosen such that they intersect each other minimally (we can picture them as chords in the $D_{w_j}$'s). 
Let $\tilde{\calD}(\psi_0\circ\phi_0)$ denote the resulting space with $\alpha$ and $\beta$ curves.
\end{construction}
\begin{example}\label{GluingConstructionExample}
See Figure \ref{GluingConstructionExampleFig} for the construction of $\tilde{\calD}(\psi_0\circ\phi_0)$ when $\phi_0 = \psi_0$ is a Dehn twist on the once punctured torus.
\end{example}

\begin{figure}
\begin{center}
 \includegraphics[scale=.6]{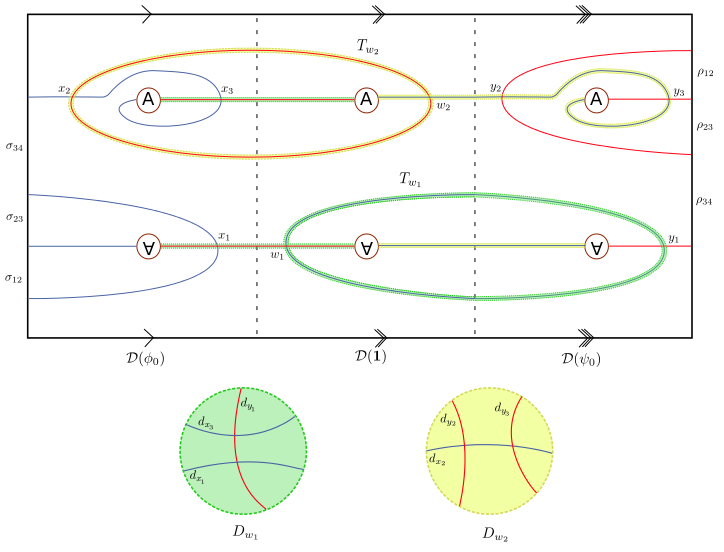}
\end{center}
\caption{The construction of $\tilde{\calD}(\psi_0\circ\phi_0)$ when $\phi_0 = \psi_0$ is a Dehn twist on the once punctured torus. Note that
here the orientations of $\calD(\phi_0)$ and $\calD(\psi_0)$ are the opposite of the standard orientation of the plane.}
\label{GluingConstructionExampleFig}
\end{figure}

\begin{lemma}\label{arc segments are disjoint}
The $\beta$ arc segments $\{d_{x_i}\}$ (resp. $\alpha$ arc segments $\{d_{y_k}\}$) in Construction \ref{Gluing Construction} are pairwise disjoint.
Two arcs in $\{d_{x_i}\}$ and $\{d_{y_k}\}$ respectively intersect if they lie in the same disk of $\{D_{w_j}\}$.
\end{lemma}
\begin{proof}
This follows more or less directly from the construction.
\end{proof}

\begin{lemma}
For some $\xi_0$ representing an element $\xi \in MCG_0(F(\calZ))$, there is a diffeomorphism
$\tilde{\calD}(\psi_0\circ\phi_0) \cong \calD(\xi_0)$ sending $\alpha$ curves to $\alpha$ curves and $\beta$ curves to $\beta$ curves.
\end{lemma}
\begin{proof}
We first compute the genus of $\tilde{\calD}(\psi_0\circ\phi_0)$ as follows.
Let $\chi(\cdot)$ denote Euler characteristic and let $g'$ be the genus of $\calD(\phi_0) \cup_{\sigma} -\calD(\mathbb{I}) \cup_{\rho} \calD(\psi_0)$.
Then we have 
\begin{align*}
2 - 2g' - 2b = \chi(\calD(\phi_0) \cup_{\sigma} -\calD(\mathbb{I}) \cup_{\rho} \calD(\psi_0)) = 3\chi(\calD(\mathbb{I})) = 3(2-2g-2b),
\end{align*}
and therefore $g' = 3g +2b - 2$.
Since each destabilization lowers the genus by one, and we perform one for each of the $2(g+b-1)$ $\alpha$ arcs of $\calD(\mathbb{I})$, we
conclude that the genus of $\tilde{\calD}(\psi_0\circ\phi_0)$ is $g$.

Next, observe that $\calD(\mathbb{I}) \setminus \boldsymbol{\alpha}$ and $\tilde{\calD}(\psi_0\circ\phi_0) \setminus \boldsymbol{\alpha}$ have the same
Euler characteristic. Then since $\calD(\mathbb{I}) \setminus \bold{\alpha}$ is a collection of $b$ annuli,
so is $\tilde{\calD}(\psi_0\circ\phi_0) \setminus \boldsymbol{\alpha}$, and we can find a diffeomorphism
$\tilde{\calD}(\psi_0\circ\phi_0) \setminus \boldsymbol{\alpha} \cong \calD(\mathbb{I}) \setminus \boldsymbol{\alpha}$ respecting $\alpha$ arcs and boundary segments.
This quotients to a diffeomorphism $\tilde{\calD}(\psi_0\circ\phi_0) \cong \calD(\mathbb{I})$ respecting $\alpha$ arcs and boundaries,
and the same reasoning shows that the resulting the image of the $\beta$ arcs differ from those of $\calD(\mathbb{I})$ by
some $\xi_0$.
\end{proof}

\begin{lemma}\label{uniqueness of representation}	 
The mapping class $\xi$ above is equal to $\psi \circ \phi$. 
\end{lemma}
Let $M(\tilde{\calD}(\psi_0\circ\phi_0))$ denote the $\calB(\calZ')-\calB(\calZ)$ $\calA_{\infty}$-bimodule defined by $\tilde{D}(\psi_0\circ\phi_0)$.
\begin{corollary}\label{uniquess of representation corollary}
$M(\tilde{\calD}(\psi_0\circ\phi_0))$ is $\calA_{\infty}$-homotopy equivalent to $M(\psi_0\circ\phi_0)$.
\end{corollary}
\begin{proof}[Sketch of proof of Lemma \ref{uniqueness of representation}]
Lemma \ref{uniqueness of representation} can be understood in terms of $\alpha-\beta$-bordered Heegaard diagrams.
$\calD(\phi_0)$ and $\calD(\psi_0)$ can be understood as representing mapping cylinders for $\phi$ and $\psi$ respectively,
and gluing $\calD(\phi_0)$ to $\calD(\psi_0)$ gives a representation of the mapping cylinder for $\psi\circ\phi$.
Moreover, it follows from the general theory of bordered Heegaard diagrams that performing Heegaard moves (isotopies, handleslides and (de)stabilizations) leave the 
represented 3-manifold invariant, which is why our construction of $\tilde{\calD}(\psi\circ\phi)$ by gluing, destabilizing, and replacing arc segments is successful.
The reader is encouraged to consult Section 3 of [\ref{LOT10c}] for more details.
\end{proof}

\subsection{Checking the composition behavior}\label{Checking the composition behavior}
The following definition gives the natural notion of isomorphism in the category of $\calA_{\infty}$-bimodules.
\begin{definition}\label{A infinity isomorphism}
An $\calA_{\infty}$-homomorphism $F = \{F_{i,1,j}\}$ from $\vphantom{M}_\calA M_{\calB}$ to $\vphantom{N}_\calA N_{\calB}$
 is an {\em $\calA_{\infty}$-isomorphism} if there exists an $\calA_{\infty}$-homomorphism $G = \{G_{i,1,j}\}$ from 
$\vphantom{N}_\calA N_{\calB}$ to $\vphantom{M}_\calA M_{\calB}$ 
such that $G \circ F = \mathbb{I_{M}}$ and $F \circ G = \mathbb{I_{N}}$. 
Note that in particular $F$ is an $\calA_{\infty}$-homotopy equivalence.
\end{definition}

\begin{theorem}\label{bigtheorem}
There is an $\calA_{\infty}$-isomorphism $F: M(\tilde{\calD}(\psi_0\circ\phi_0)) \rightarrow M(\phi_0) \boxtimes \left(\dd \boxtimes M(\psi_0)\right)$
with $F_{i,1,j} = 0$ unless $i = j = 0$.
\end{theorem}
\begin{corollary}\label{bigcorollary}
The bimodules $M(\psi\circ\phi)$ and $M(\phi) \boxtimes \left(\dd \boxtimes M(\psi)\right)$ are $\calA_{\infty}$-homotopy equivalent.
\end{corollary}
\begin{proof}
This is immediate from Therem \ref{isotopyinvariance} and Corollary \ref{uniquess of representation corollary}.
\end{proof}

\begin{proof}[Proof of Theorem \ref{bigtheorem}]
The morphism 
\begin{align*}
F: M(\tilde{\calD}(\psi_0\circ\phi_0)) \rightarrow M(\phi_0) \boxtimes \left(\dd \boxtimes M(\psi_0)\right)
\end{align*}
has $F_{i,1,j} := 0$ unless $i=j=0$, with $F_{0,1,0}$ defined in the following way.
Recall that $M(\tilde{\calD}(\psi_0\circ\phi_0))$ is generated by the intersection points between $\alpha$ and $\beta$ arcs in
$\tilde{\calD}(\psi_0\circ\phi_0)$. Let us denote this set by $\{z_l\}$.
Each $z_l$ occurs in some $D_{w_j}$ as $d_{x_i}\cap d_{y_k}$ for some $d_{x_i}$ and $d_{y_k}$.
In this case, following the definitions, we have $x_i \otimes (w_j \otimes y_k) \in \GEN$, 
where $\GEN$ is the canonical $\F_2$-basis for $M(\phi_0) \boxtimes \left(\dd \boxtimes M(\psi_0)\right)$.
We define $F_{0,1,0}$ by 
\begin{align*}
F_{0,1,0}(z_l) := x_i \otimes (w_j \otimes y_k). 
\end{align*}

We claim that, for each $x_i \otimes (w_j\otimes y_k) \in \GEN$, the arc segments $d_{x_i}$ and $d_{y_k}$ intersect at some point $z_l$ in some $D_{w_j}$.
It then follows that $F_{0,1,0}$ sets up a bijective correspondence between $\{z_l\}$ and $\GEN$,
and we extend it by linearity to a $\kk$-bimodule isomorphism.

We let 
\begin{align*}
G: M(\phi_0) \boxtimes \left(\dd \boxtimes M(\psi_0)\right) \rightarrow M(\tilde{\calD}(\psi_0\circ\phi_0))
\end{align*}
be the morphism defined by $G_{0,1,0} := F_{0,1,0}^{-1}$ and $G_{i,1,j} := 0$ otherwise.
Our claim is that $F$ and $G$ are inverse $\calA_{\infty}$-isomorphisms.
Using the definition of morphism composition (see (\ref{A infinity homomorphism composition})), we see that $F$ and $G$ satisfy the conditions of Definition \ref{A infinity isomorphism} if
$F$ and $G$ are $\calA_{\infty}$-homomorphisms. 
This is in turn equivalent to the following reduction of (\ref{A infinity homomorphism relations}):
\begin{align*}
\xymatrix
{
\ar@/_/@{=>}[ddr]&\ar@{-->}[d]&\ar@/^/@{=>}[ddl]&&\ar@/_/@{=>}[dr]&\ar@{-->}[d]&\ar@/^/@{=>}[dl]&\\
&F_{0,1,0}\ar@{-->}[d]&&+&&m\ar@{-->}[d]&&=\;0.\\
&m\ar@{-->}[d]&&&&F_{0,1,0}\ar@{-->}[d]&&\\
&&&&&&&\\
} 
\end{align*}
That is, we must show that $F_{0,1,0}$ sets up bijective correspondence between the $m$ relations of either bimodule. 
We verify the above equation by breaking it into two steps. $\textbf{Direction 1}$ shows that for each nontrivial
summand of $m$ in $M(\phi_0) \boxtimes \left(\dd \boxtimes M(\psi_0)\right)$ there is the corresponding summand of $m$ in $M(\tilde{\calD}(\psi_0\circ\phi_0))$,
and $\textbf{Direction 2}$ shows that for each nontrivial summand of $m$ in $M(\tilde{\calD}(\psi_0\circ\phi_0))$
there is the corresponding summand of $m$ in $M(\phi_0) \boxtimes \left(\dd \boxtimes M(\psi_0)\right)$.

\sss
  
\noindent $\textbf{Direction 1}$:

Let us consider $m(\sigma_m,...,\sigma_1,x_a\otimes(w_b \otimes y_c),\rho_1,...,\rho_n)$
for $x_a \otimes (w_b \otimes y_c) \in \GEN$.
To see what a summand means in terms of $M(\phi_0),\dd$, and $M(\psi_0)$, 
observe that the structure map $\delta_{DA}$ on $\dd \boxtimes M(\psi_0)$ is given by
\begin{align*}
\xymatrix@=.9em
{
&\ar@{--}[dd]&\ar@/^/@{=>}[ddl]    &&    \ar@{--}[dddd]&\ar@{--}[dddd]    &&    &\ar@{--}[d]&\ar@{--}[dd]&\ar@/^/@{=>}[ddl]  &&  &&\ar@{--}[d]&\ar@{--}[dd]&\ar@{=>}[d]	&\\
&&    &&    &    &&    &\delta_{DD}\ar@{--}[ddd]\ar@/_/@{=>}[dl]\ar@/^/@{=>}[dr]&&  &&  &&\delta_{DD}\ar@/_/@{=>}[dl]\ar@{--}[d]\ar@/^/@{=>}[dr]&&\Delta_2\ar@/^/@{=>}[dl]\ar@/^/@{=>}[ddl]	&\\
&\delta_{DA}\ar@/_/@{=>}[ddl]\ar@{--}[dd]&    &=&    &    &+&    \Pi\ar[dd]&&m\ar@{--}[dd]&  &+&  &\Pi\ar@/_/[ddl]&\delta_{DD}\ar@/_/@{=>}[dl]\ar@{--}[dd]\ar@/^/@{=>}[dr]&m\ar@{--}[d]&	& +\;...\\
&&    &&    &    &&    &&&  &&  &\Pi\ar@/_/[dl]&&m\ar@{--}[d]&	&\\
&&    &&    &    &&    &&&  &&  &&&&	&\\
}
\end{align*}

Thus each summand of $m(\sigma_1,...,\sigma_m,x_a\otimes(w_b \otimes y_c),\rho_1,...,\rho_n)$ comes from a diagram of the form
\begin{align}\label{case 3}
\xymatrix@C=.5em@R=1em
{
\sigma_m\ar@/_/[ddddddddrrr]&...\ar@/_/@{=>}[ddddddddrr]&\sigma_1\ar@/_/[ddddddddr]&x_a\ar@{-->}[dddddddd]    &&w_b\ar@{-->}[d]&y_c\ar@{-->}[dd]&		\rho_{t_0+1}\ar@/^/[ddl]&...\ar@/^/@{=>}[ddll]&\rho_{t_1}\ar@/^/[ddlll]&\rho_{t_1+1}\ar@/^/[dddllll]&...\ar@/^/@{=>}[dddlllll]&\rho_{t_2}\ar@/^/[dddllllll]&\ar@{.}[rr]&\ar@/^/@{=>}[dddddllllllll]&\rho_{t_{r-1}+1}\ar@/^/[dddddddlllllllll]&...\ar@/^/@{=>}[dddddddllllllllll]&\rho_{t_r}\ar@/^/[dddddddlllllllllll] \\
&&&    &&\delta_{DD}^{s_1}\ar@{-->}[d]\ar@/_/@{=>}[dl]\ar@/^/@{=>}[dr]&&		\\
&&&    &\Pi\ar@/_/[ddddddl]&\delta_{DD}^{s_2}\ar@{-->}[d]\ar@/_/@{=>}[dl]\ar@/^/@{=>}[dr]&m\ar@{-->}[d]&		\\
&&&    &\Pi\ar@/_/[dddddl]&\ar@{.}[dd]&m\ar@{-->}[d]&		\\
&&&    &\ar@{.}[dd]&\ar@/_/@{=>}[dl]\ar@/^/@{=>}[dr]&\ar@{.}[dd]&		\\
&&&    &\ar@{=>}@/_/[dddl]&\ar@{-->}[d]&&		\\
&&&    &&\delta_{DD}^{s_r}\ar@{-->}[ddd]\ar@/_/@{=>}[dl]\ar@/^/@{=>}[dr]&\ar@{-->}[d]&		\\
&&&    &\Pi\ar@/_/[dl]&&m\ar@{-->}[dd]&		\\
&&&m\ar@{-->}[d]    &&&&		\\
&&&&&&&
}
\end{align}
where $0 = t_0 \leq t_1 \leq ... \leq t_r = n$, $s_1,...,s_r \geq 0$, and $r \geq 0$.

In particular, the case $r = 0$ manifests itself as
\begin{align}\label{case 1}
\xymatrix@=3em
{
\sigma_m\ar@/_/[drrr]&...\ar@{=>}@/_/[drr]&\sigma_1\ar@/_/[dr]&x_{a}\ar@{-->}[d]&w_b\ar@{-->}[dd]&y_c\ar@{-->}[dd]\\
&&&m\ar@{-->}[d]&&\\
&&&&&
}
\end{align}
On the other hand, if $r >0$ and any of $s_1,...,s_r$ is zero, then we must have $r = 1$ and $s_1 = 0$, 
since otherwise the $m$ in the left of (\ref{case 1}) would be trivial. Then this case corresponds to a diagram of the form
\begin{align}\label{case 2}
\xymatrix@=2em
{
x_{a}\ar@{-->}[d]&\ar^{\mathbf{1_{\calB(\calZ)}}}@/^/[dl]&w_b\ar@{-->}[dd]&y_c\ar@{-->}[d] &\rho_1\ar@/^/[dl]&...\ar@{=>}@/^/[dll]&\rho_n\ar@/^/[dlll]\\
m\ar@{-->}[d]&&&m\ar@{-->}[d]\\
&&&
}
\end{align}

Accordingly, we break our work into three cases:
\begin{itemize}
 \item $\textbf{Case 1}$: $r = 0$
 \item $\textbf{Case 2}$: $r = 1$ and $s_1 = 0$
 \item $\textbf{Case 3}$: $r \geq 1$ and $s_1,...,s_r \geq 1$
\end{itemize}

\sss

For convenience let $\simt: \bdy\boldsymbol{D} \rightarrow \bdy\boldsymbol{T}$ denote the diffeomorphism
defining $\sim$, and let $\simd: \bdy\boldsymbol{T} \rightarrow \bdy\boldsymbol{D}$ denote its inverse.
In what follows $d(\cdot)$ and $t(\cdot)$ will always be corresponding segments as in Construction (\ref{Gluing Construction}),
and similarly for $\DD(\cdot)$ and $\TT(\cdot)$.
We will use $\Dom(\cdot)$ to denote the domain of a map.

\sss

\noindent $\textbf{Case 1}$:
For a summand $x_{a'} \otimes (w_{b'}\otimes y_{c'})$ of $m(\sigma_m,...,\sigma_1,x_a\otimes(w_b \otimes y_c),\rho_1,...,\rho_n)$ 
coming from $(\ref{case 1})$, it is clear that $w_{b'} = w_b,\;y_{c'} = y_c$, and $n = 0$.
The $m$ in $(\ref{case 1})$ implies the existence of a polygon
$P_L: \D^2 \rightarrow \calD(\phi_0)$ from $x_a$ to $x_{a'}$ through $(\sigma_1,...,\sigma_m)$ and $()$.
A little thought shows that the connected components of $P_L^{-1}(\boldsymbol{T})$ 
are neighborhoods $\{\ova{A}_l\}$ of $\alpha$ segments in $\Int(\D^2)$ with both endpoints on the left side of $\D^2$, as well as a neighborhood $A_\init$ of the right side of $\D^2$.
Let $\calP_L$ denote the restriction of $P_L$ to $P_L^{-1}(\calD(\phi_0)\setminus\boldsymbol{T})$.
Our plan is to upgrade $\calP_L$ to a polygon $P_{\text{tot}}: \D^2 \rightarrow \tilde{\calD}(\psi_0\circ\phi_0)$ from $F^{-1}_{0,1,0}(x_a \otimes (w_b \otimes y_c))$ 
to $F^{-1}_{0,1,0}(x_{a'} \otimes (w_b \otimes y_c))$ through $(\sigma_1,...,\sigma_m)$ and $()$, and we do this by gluing regions of disks to $\calP_L$.
%

Consider the $\beta$ segments $d_{x_a},d_{x_{a'}} \in D_{w_b}$ and the $\alpha$ segment $d_{y_c} \in D_{w_b}$.
By Lemma \ref{arc segments are disjoint} we know that $d_{x_a}$ and $d_{x_{a'}}$ are disjoint and both intersect $d_{y_c}$.
Observe that $\calP_L$ maps $\bdy \calP_L \cap \bdy A_\init$ diffeomorphically to a segment $s$ of $\bdy T_{w_b}$ connecting an endpoint of $t_{x_a}$ to an endpoint of $t_{x_{a'}}$,
and it corresponds under $\sim$ to a segment $\simd(s)$ of $\bdy D_{w_b}$ connecting an endpoint of $d_{x_a}$ to an endpoint of $d_{x_{a'}}$.
Together $d_{x_a},d_{x_{a'}},d_{y_c}$, and $\simd(s)$ define a rectangular region $R_\init$ of $D_{w_b}$.
We proceed by gluing $R_\init$ to $\calP_L$ along $R_\init \cap \bdy D_{w_b}$.

Similarly, the intersection of each $\ova{A}_l$ with the left side of $P_L$ is two nonintersecting $\beta$ segments $t^-(\ova{A}_l),t^+(\ova{A}_l) \in \{t_{x_i}\}$.
The corresponding $\beta$ segments $d^-(\ova{A}_l),d^+(\ova{A}_l) \in \{d_{x_i}\}$ define a rectangular region $\ova{R}_l$ of some disk $D(\ova{R}_l) \in \{D_{w_j}\}$.
We proceed by gluing $\ova{R}_l$ to $\calP_L$ along $\ova{R}_l \cap \bdy D(\ova{R}_l)$.

The result after these gluings is the desired polygon $P_\tot$.

\sss

\noindent $\textbf{Case 2}$:
This case follows similarly to Case 1, except with the roles of the left and right sides switched.

\sss

\noindent $\textbf{Case 3}$:
Finally, suppose we have a summand $x_{a'} \otimes (w_{b'}\otimes y_{c'})$ coming from $(\ref{case 3})$ with $r \geq 1$ and $s_1,...,s_r \geq 1$. 
Assume that for each $i$ the contributing output of $\delta^{s_i}_{DD}$ is 
$\xi^{i,1}\otimes ... \otimes \xi^{i,s_i}$ on the left and $\xi^{i,1'}\otimes ... \otimes \xi^{i,s_i'}$ on the right, 
where $\xi^{i,j}$ and $\xi^{i,j'}$ are corresponding short chords in the sense of Section \ref{Various types of bimodules}.
Then we have
\begin{itemize}
 \item a polygon $P_L: \D^2 \rightarrow \calD(\phi_0)$ from $x_a$ to $x_{a'}$ through $(\sigma_1,...,\sigma_m)$ and $(\Pi(\xi^{1,1},...,\xi^{1,s_1}),...,\Pi(\xi^{r,1},...,\xi^{r,s_r}))$,
 \item a compact connected component $P^{i,j}_M$ of $\calD(\mathbb{I})\setminus (\boldsymbol{\alpha}\cup\boldsymbol{\beta})$ 
\footnote{Here $P^{i,j}_M$ is really the closure of a component. We will often think of $P^{i,j}_M$ as the identity map on itself.}
containing $\xi^{i,j}$ and $\xi^{i,j'}$ for 
 each $1 \leq i \leq r$ and $1 \leq j \leq s_i$, and 
 \item a polygon $P^i_R: \D^2 \rightarrow \calD(\psi_0)$ through $(\xi^{i,1'},...,\xi^{i,s_i'})$ and $(\rho_{t_{i-1}+1},...,\rho_{t_i})$ for each $1 \leq i \leq r$.
\end{itemize}

We begin by analyzing the preimage of $\TT$ in each $P_R^i$ and $P_L$. 
Since the left side of $P_R^i$ only includes short chords, the components of $(P_R^i)^{-1}(\TT)$ are neighborhoods $\{\una{A}^i_l\}$ of 
$\beta$ segments in $\Int(P_R^i)$ with both endpoints on the right side of $P_R^i$ and neighborhoods of $\beta$ segments of the left side of $P_R^i$.
On the other hand, the components of $P_L^{-1}(\TT)$ are 
\begin{itemize}
 \item neighborhoods $\{\ova{A}_l\}$ of $\alpha$ segments in $\Int(P_L)$ with both endpoints on the left side of $P_L$,
  \item neighborhoods $\{\ovl{A}_l\}$ of $\alpha$ segments in $\Int(P_L)$ with one endpoint on the left side of $P_L$ and one endpoint on the right side of $P_L$, and
 \item neighborhoods of $\alpha$ segments of the right side of $P_L$.
\end{itemize}
\begin{remark}\label{exludebothsidedarcs}
Notice that we have excluded neighborhoods of $\alpha$ segments in $\Int(P_L)$ with endpoints on each side of $P_L$, since this would imply
a component of $F \setminus \boldsymbol{\alpha}$ whose boundary lies entirely in $\cup_i S_i^-$, contradicting 
the conditions given in Section \ref{arc diagrams and the algebra B(Z)}.
\end{remark}

Let $\calP_L,\calP_M^{i,j}$, and $\calP_R^i$ denote the restrictions of $P_L,P_M^{i,j}$, and $P_R^i$ to 
$P_L^{-1}(\calD(\phi_0)\setminus\TT),(P_M^{i,j})^{-1}(\calD(\mathbb{I})\setminus \TT)$, and $(P_R^i)^{-1}(\calD(\psi_0)\setminus \TT) $ respectively.
We proceed by gluing each $\calP^{i,j}_M$ along its $\xi^i_j$ side to $\calP_L$ and along its $\xi^{i'}_j$ side to $\calP^i_R$ in the obvious way implied by
$(\ref{case 3})$, the gluing taking place along (slight reductions of) short chords.
The result is a map $\calP_\tot$ into 
$(\calD(\phi_0) \cup_{\sigma} -\calD(\mathbb{I}) \cup_{\rho} \calD(\psi_0)) \setminus \DD$.
Our plan is to upgrade $\calP_\tot$ to a polygon $P_{\tot}: \D^2 \rightarrow \tilde{\calD}(\psi_0\circ\phi_0)$ 
from $F^{-1}_{0,1,0}(x_{a} \otimes (w_b \otimes y_c))$ to $F^{-1}_{0,1,0}(x_{a'} \otimes (w_{b'} \otimes y_{c'}))$
through $(\sigma_1,...,\sigma_m)$ and $(\rho_1,...,\rho_n)$
by gluing in various regions of the disks $\{D_{w_j}\}$.

Firstly, we have already seen in cases 1 and 2 how to glue in regions $\una{R}_l^i$ and $\ova{R}_l$ corresponding to each $\una{A}_l^i$ and $\ova{A}_l$ respectively.

Denote the resulting map after this gluing by $\calP_\tot'$.
Observe that, for each $l$, $\bdy \Dom(\calP_\tot') \setminus \bdy \ovl{A}_l$ has two components, one of which intersects the left side of $P_L$ and the right side of each $P_R^i$.
Let us denote the union of the other component and $\bdy\Dom(\calP_\tot') \cap \bdy \ovl{A}_l$ by $\BS(\ovl{A}_l)$.
For coherence we break our remaining gluing into four steps.
\begin{enumerate}

 \item The intersection of each $\bdy \ovl{A}_l$ with the left side of $P_L$ is a $\beta$ segment $t(\ovl{A}_l) \in \{t_{x_i}\}$,
and $\calP_\tot'$ maps $\BS(\ovl{A}_l)$ diffeomorphically onto a segment $\calP_\tot'(\BS(\ovl{A}_l)$) of some $\bdy T(\ovl{A}_l) \in \{\bdy T_{w_j}\}$
with endpoints $\bdy t(\ovl{A}_l)$.
Together $d(\ovl{A}_l)$ and $\simd(\calP_\tot'(\BS(\ovl{A}_l)))$ define a bigonal region $\ovl{R}_l$ of $D(\ovl{A}_l)$, which we glue along $\bdy \ovl{R}_l \cap \bdy D(\ovl{A}_l)$
to $\BS(\ovl{A}_l)$.
 
 \item Now let $\calP_\tot''$ denote the resulting map, and observe that $\Dom(\calP_\tot'')$ is simply-connected.
Let $\gamma_L(P_L)$ and $\gamma_R(P_R^i)$ denote the left and right sides of $P_L$ and $P_R^i$ respectively.
Let $C_\init,C_1,...,C_{r-1},C_\term$ denote the components of
\begin{align*}
\Dom(\calP_\tot'') \setminus \left(\gamma_L(P_L) \cup_{i} \gamma_R(P_R^i) \cup_{l} \bdy \una{R}_l \cup_l\bdy \ova{R}_l \cup_l\bdy \ovl{R}_l\right),
\end{align*}
ordered such that $C_\init$ intersects $\calP_R^1$, $C_i$ intersects $\calP_R^{i}$ and $\calP_R^{i+1}$ for $1 \leq i \leq r-1$, and $C_\term$
intersects $\calP_R^r$.
Then for $1 \leq i \leq r-1$, $\calP_\tot''$ maps $C_i$ diffeomorphically onto a segment $\calP_\tot''(C_i)$ of some $\bdy T(C_i) \in \{\bdy T_{w_j}\}$ with endpoints
$\bdy t(C_i) \in \{t_{y_k}\}$. Together $d(C_i)$ and $\simd(\calP_\tot''(C_i))$ define a bigonal region $\unl{R}_l$ of $D(C_i)$, which we glue
along $\bdy \unl{R}_l \cap \bdy D(C_i)$ to $C_i$.

\item Finally, we consider $C_\init$ and $C_\term$. $\calP_\tot''$ maps $C_\init$ diffeomorphically onto a segment $\calP_\tot''(C_\init)$ of $\bdy T_{w_b}$
with endpoints in $\bdy t_{x_a}$ and $\bdy t_{y_c}$ respectively.
Then $\simd(\calP_\tot''(C_\init))$ together with $d_{x_a}$ and $d_{y_c}$ define a triangular region $R_\init$ of $D_{w_b}$,
which we glue along $\bdy R_\init \cap \bdy D_{w_b}$ to $C_\init$.
The vertex $d_{x_a} \cap d_{y_c}$ is of course none other than $F^{-1}_{0,1,0}(x_a\otimes(w_b\otimes y_c))$.
We proceed similarly for $C_\term$, gluing a triangular region $R_\term$ of $D_{w_{b'}}$ along 
$\bdy R_\term \cap \bdy D_{w_{b'}}$ to $C_\term$.

\end{enumerate}
The result after all of this gluing is the promised polygon $P_{\text{tot}}$.

\sss

\sss

\noindent $\textbf{Direction 2}$:  

Now let $P_{\text{tot}}: \D^2 \rightarrow \tilde{\calD}(\psi_0\circ\phi_0)$ be a polygon 
from $F^{-1}_{0,1,0}(x_{a} \otimes (w_b \otimes y_c))$ to $F^{-1}_{0,1,0}(x_{a'} \otimes (w_{b'} \otimes y_{c'}))$
through $(\sigma_1,...,\sigma_m)$ and $(\rho_1,...,\rho_n)$. Our task is to reverse the process of \textbf{Direction 1} by 
breaking up $P_{\text{tot}}$ into smaller components which fit together as in (\ref{case 3}).
We first examine the components of $P_{\text{tot}}^{-1}(\boldsymbol{D})$.

\begin{lemma}\label{preimageofTinP}
The components of $P_{\text{tot}}^{-1}(\boldsymbol{D})$ consist of:
\begin{itemize}
 \item a domain $R_{\text{init}} \subset D_{w_b}$ containing $-i$, and possibly a distinct domain $R_{\text{term}} \subset D_{w_{b'}}$ containing $i$
  \item domains $R_l \cong D(R_l) \in \{D_{w_j}\}$ in the interior of $\Dom(P_{\text{tot}})$ 
 \item domains $\ovl{R}_l \subset D(\ovl{R}_l) \in \{D_{w_j}\}$ such that $\ovl{R}_l \cap \bdy(\Dom(P_{\text{tot}}))$ is a single segment of the {\em left} side of $\Dom(P_{\text{tot}})$
 \item domains $\unl{R}_l \subset D(\unl{R}_l) \in \{D_{w_j}\}$ such that $\unl{R}_l \cap \bdy(\Dom(P_{\text{tot}}))$ is a single segment of the {\em right} side of $\Dom(P_{\text{tot}})$
 \item domains $\ova{R}_l \subset D(\ova{R}_l) \in \{D_{w_j}\}$ such that $\ova{R}_l \cap \bdy(\Dom(P_{\text{tot}}))$ is two disjoint segments of the {\em left} side of $\Dom(P_{\text{tot}})$
 \item domains $\una{R}_l \subset D(\una{R}_l) \in \{D_{w_j}\}$ such that $\una{R}_l \cap \bdy(\Dom(P_{\text{tot}}))$ is two disjoint segments of the {\em right} side of $\Dom(P_{\text{tot}})$
\end{itemize}
where (by abuse of notation) we have used $\subset$ to denote an embedding under $P_{\text{tot}}$.
Moreover, we have one of the following:
\begin{itemize}
 \item \textbf{Case 1}: $R_{\text{init}} = R_{\text{term}}$ is $\Dom(P_{\text{tot}})$ minus a half disk with equator along
the {\em left} side of $\Dom(P_{\text{tot}})$,
 \item \textbf{Case 2}: $R_{\text{init}} = R_{\text{term}}$ is $\Dom(P_{\text{tot}})$ minus a half disk with equator along
the {\em right} side of $\Dom(P_{\text{tot}})$, or 
 \item \textbf{Case 3}: $R_{\text{init}}$ and $R_{\text{term}}$ are distinct, and $R_\init \cap \bdy(\Dom(P_\tot))$ and $R_\term \cap \bdy(\Dom(P_\tot))$ are both connected.
\end{itemize}
\end{lemma}
\begin{proof}
The first part of the lemma follows essentially from Lemma \ref{arc segments are disjoint}.
In particular, a component of $P_\tot^{-1}(\DD)$ whose intersection with $\bdy(\Dom(P_\tot))$ has components on both sides of $\Dom(P_\tot)$ would imply a pair of nonintersecting 
arcs $d_{x_i}$ and $d_{y_k}$ in the same disk $D_{w_j}$.
On the other hand, a component of $P_\tot^{-1}(\DD)$ whose intersection with $\bdy(\Dom(P_\tot))$ has 
more than two components on the same side of $\Dom(P_\tot)$
would imply three segments of $\{d_{x_i}\}$ (resp. $\{d_{y_k}\}$) 
in a disk $D_{w_j}$ with the wrong combinatorics (i.e. disagreeing with Lemma \ref{arc segments are disjoint}).
Namely, no element of $\{d_{y_l}\}$ (resp. $\{d_{x_i}\}$) could intersect all three in minimal position.

The second part of the lemma follows similarly.
\end{proof}

\begin{example}
See Figure \ref{PolygonConstruction2} for an example of \textbf{Case 3}. 
\end{example}

We consider Cases 1 - 3 from Lemma \ref{preimageofTinP} separately.
Let $\alpha(x_i)$ denote the $\alpha$ arc of $\calD(\phi_0)$ containing $x_i$, and let $\beta(y_k)$ denote
the $\beta$ arc of $\calD(\psi_0)$ containing $y_k$.

\sss

\noindent \textbf{Case 1}: In this case the entire right side of $P_{\text{tot}}$ is a segment of some $\alpha$ segment from $\{d_{y_k}\}$,
and therefore $w_{b'} = w_{b}$, $y_{c'} = y_{c}$, and $n = 0$. 
Let $\calP_L$ denote the restriction of $P_\tot$ to $\Dom(P_\tot) \setminus P_\tot^{-1}(\DD)$.
Then $\calP_L$ maps $\bdy \calP_L \cap \bdy R_\init$ diffeomorphically to a segment $\calP_L(\bdy \calP_L \cap \bdy R_\init)$ of $\bdy D_{w_b}$
with endpoints in $d_{x_a}$ and $d_{x_{a'}}$ respectively.
Together $\simt(\calP_L(\bdy \calP_L \cap \bdy R_\init))$, $t_{x_a}$, $t_{x_{a'}}$, and $\alpha(x_a) = \alpha(x_{a'})$ define
a rectangular region $A_\init$ of $T_{w_b}$. 
We proceed by gluing as follows:
\begin{itemize}
 \item Glue $A_\init$ to $\calP_L$ along $\bdy A_\init \cap \bdy T_{w_b}$.

 \item For each $\ova{R}_l$, the intersection $\ova{R}_l \cap \bdy(\Dom(P_{\text{tot}}))$ is two $\beta$ segments $d^-(\ova{R}_l),d^+(\ova{R}_l) \in \{d_{x_{i}}\}$.
The corresponding $\beta$ segments $t^-(\ova{R}_l),t^+(\ova{R}_l) \in \{t_{x_{i}}\}$, together with the two components of 
$\simt(\calP_L(\ova{R}_l \cap \calP_L))$, define a rectangular region $\ova{A}_l$ of $T(\ova{R}_l)$, which we glue to 
$\calP_L$ along $\bdy \ova{R}_l$.
\end{itemize}

Let $P_L$ denote the result after the above gluing. The following lemma shows that $P_L$ is a polygon
in $\calD(\phi_0)$ from $x_a$ to $x_{a'}$ through $(\sigma_1,...,\sigma_m)$ and $()$, as in (\ref{case 1}).

\begin{lemma}
$\{R_l\} = \{\ovl{R}_l\} = \{\unl{R}_l\} = \{\una{R}_l\} = \nil$.
\end{lemma}
\begin{proof}
The fact that $\{\unl{R}_l\}$ and $\{\una{R}_l\}$ are empty is immediate, since the entire right side of $P_\tot$ is contained in $R_\init$.
We note that $\calP_L$ has image in $\calD(\phi_0)$, and therefore $\{R_l\}$ and $\{\ovl{R}_l\}$ must be empty as well.
\end{proof}

\sss

\noindent \textbf{Case 2}: This case follows similarly to Case 1, except with the roles of the left and right sides switched.

\sss

\noindent \textbf{Case 3}: In this case, let $\boldsymbol{\tilde{\rho}}$ and $\boldsymbol{\tilde{\sigma}}$ 
denote the images of the $\rho$ and $\sigma$ boundaries of $\calD(\mathbb{I}) \setminus \TT$ in $\tilde{\calD}(\psi_0\circ\phi_0)$. 
Observe that each component of $P_\tot^{-1}(\calD(\mathbb{I}))$ is a rectangular region $\calP_M^{i,j}$ 
(with the indices $i,j$ to be defined shortly)
of $\Dom(P_\tot)$. Let $\xi^{i,j} = \calP_M^{i,j} \cap \boldsymbol{\tilde{\rho}}$ and $\xi^{i,j'} = \calP_M^{i,j} \cap \boldsymbol{\tilde{\sigma}}$.
Let $\calP_\tot$ denote the restriction of $P_\tot$ to $\Dom(P_\tot) \setminus P_\tot^{-1}(\DD)$.

\begin{lemma}
$\{R_l\} = \nil$. 
\end{lemma}
\begin{proof}
Let $r$ and $s$ denote the number of components of 
$\Dom(\calP_\tot) \setminus \cup_{i,j}\calP_M^{i,j}$
and $P_{\tot}^{-1}(\boldsymbol{\tilde{\rho}} \cup \boldsymbol{\tilde{\sigma}})$ respectively.
We first claim that each component of $\Dom(\calP_\tot) \setminus \cup_{i,j}\calP_M^{i,j}$
intersects $\bdy \Dom(P_\tot)$. This is because the image of such a component would violate
the conditions of Section \ref{arc diagrams and the algebra B(Z)}, as in Remark \ref{exludebothsidedarcs}.
Then a little thought shows that we must have
\begin{align*}
r &\leq  |\{\ovl{R}_l\}| + |\{\unl{R}_l\}| + 2 \\
r &= s/2 - |\{R_l\}| + 1\\
s &= 2 + 2|\{\ovl{R}_l\}| + 2|\{\unl{R}_l\}| + 4|\{R_l\}|.
\end{align*}
These combine to yield
\begin{align*}
1 + |\{\ovl{R}_l\}| + |\{\unl{R}_l\}| + 2|\{R_l\}| - |\{R_l\}| + 1 \leq |\{\ovl{R}_l\}| + |\{\unl{R}_l\}| + 2,
\end{align*}
i.e. 
\begin{equation*}
|\{R_l\}| \leq 0.\qedhere 
\end{equation*}
\end{proof} 

Now let $\calP_L$ denote the restriction of $\calP_\tot$ to the union of the components of $\Dom(\calP_\tot) \setminus \cup_{i,j}\calP_M^{i,j}$
which intersect the left side of $\Dom(P_\tot)$.
Similarly, let $\calP_R^1,...,\calP_R^r$ denote the restrictions of $\calP_\tot$ to the components of $\Dom(\calP_\tot) \setminus \cup_{i,j}\calP_M^{i,j}$
which intersect the right side of $\Dom(P_\tot)$ (ordered from $-i$ to $i$).
Then $\calP_M^{i,j}$ is the $j$th element of $\{\calP_M^{i,j}\}$ intersecting $\calP_R^i$.

By gluing regions of the tori $\{T_{w_j}\}$, we will upgrade 
\begin{itemize}
 \item $\calP_L$ to a polygon $P_L: \D^2 \rightarrow \calD(\psi_0)$ from $x_a$ to $x_{a'}$ through
$(\sigma_1,...,\sigma_m)$ and $(\Pi(\xi^{1,1},...,\xi^{1,s_1}),...,\Pi(\xi^{r,1},...,\xi^{r,s_r}))$

 \item each $\calP_M^{i,j}$ to a connected component $P_M^{i,j}$ of $\calD(\mathbb{I}) \setminus (\boldsymbol{\alpha}\cup\boldsymbol{\beta})$ 
defined by $\xi^{i,j}$ and $\xi^{i,j'}$, and 

 \item each $\calP_R^i$ to a polygon $P_R^i: \D^2 \rightarrow \calD(\psi_0)$ through
$(\xi^{i,1'},...,\xi^{i,s_i'})$ and $(\rho_{t_{i-1}+1},...,\rho_{t_i})$ for each $1 \leq i \leq r$,
\end{itemize}
where $0 = t_0 \leq t_1 \leq ... \leq t_r = n$.

The gluing to $\calP_L$ is as follows:
\begin{itemize}
 \item For each $\ovl{R}_l$, the intersection $\ovl{R}_l \cap \bdy(\Dom(P_{\text{tot}}))$ maps to $\beta$ segment $d(\ovl{R}_l) \in \{d_{x_{i}}\}$.
Its corresonding $\beta$ segment $t(\ovl{R}_l) \in \{t_{x_{i}}\}$,
together with the two components of $\simt(P_\tot(\ovl{R}_l \cap \Dom(\calP_L))) \cap \calD(\phi_0)$ define a rectangular region $\ovl{A}_l$
of $T(\ovl{R}_l) \cap \calD(\phi_0)$, which we glue to $\calP_L$ along $\simt(P_\tot(\ovl{R}_l \cap \Dom(\calP_L))) \cap \calD(\phi_0)$.

 \item For each $\unl{R}_l$, the intersection $T(\unl{R}_l) \cap \calD(\phi_0)$ is split in half by $\boldsymbol{\alpha}$.
We glue the half containing $\simt(P_\tot(\unl{R}_l \cap \Dom(\calP_L))) \cap \calD(\phi_0)$ to $\calP_L$
along $\simt(P_\tot(\unl{R}_l \cap \Dom(\calP_L))) \cap \calD(\phi_0)$.

 \item For each $\ova{R}_l$, the intersection $\ova{R}_l \cap \bdy\Dom(P_{\text{tot}})$ maps to two $\beta$ segments $d^-(\ova{R}_l),d^+(\ova{R}_l) \in \{d_{x_{i}}\}$.
The corresponding $\beta$ segments $t^-(\ova{R}_l),t^+(\ova{R}_l) \in \{t_{x_{i}}\}$, 
together with the two components of $\simt(P_\tot(\ova{R}_l \cap \Dom(\calP_L)))$,
define a rectangular region $\ova{A}_l$ of $T(\ova{R}_l)$. We glue $\ova{A}_l$ to $\calP_L$ along $\simt(P_\tot(\ova{R}_l \cap \Dom(\calP_L)))$.

 \item The $\beta$ segment $t_{x_a}$, together with $\alpha(x_a)$ and $\simt(P_\tot(R_\init \cap \Dom(\calP_L))) \cap \calD(\phi_0)$,
define a rectangular region of $T_{w_b} \cap \calD(\phi_0)$,
which we glue to $\calP_L$ along $\simt(P_\tot(R_\init \cap \Dom(\calP_L))) \cap \calD(\phi_0)$.

Similarly, the $\beta$ segment $t_{x_{a'}}$, together with $\alpha(x_{a'})$ and $\simt(P_\tot(R_\term \cap \Dom(\calP_L))) \cap \calD(\phi_0)$,
define a rectangular region of $T_{w_{b'}} \cap \calD(\phi_0)$,
which we glue to $\calP_L$ along $\simt(P_\tot(R_\term \cap \Dom(\calP_L))) \cap \calD(\phi_0)$.
\end{itemize}

Similarly, the gluing to the $\calP_R^{i}$'s is as follows:
\begin{itemize}
\item For each $\ovl{R}_l$ intersecting $\calP_R^{i,j}$, the intersection $T(\ovl{R}_l) \cap \calD(\psi_0)$ is split in half by $\boldsymbol{\beta}$.
We glue the half containing $\simt(P_\tot(\ovl{R}_l \cap \Dom(\calP_R^i))) \cap \calD(\psi_0)$ to $\calP_R^i$
along $\simt(P_\tot(\ovl{R}_l \cap \Dom(\calP_R^i))) \cap \calD(\psi_0)$.
 \item For each $\unl{R}_l$ intersecting $\calP_R^{i}$ and $\calP_R^{i+1}$, the intersection
$T(\unl{R}_l) \cap \calD(\psi_0)$ is split in half by $\boldsymbol{\beta}$.
We glue these two halves to $\calP_R^i$ and $\calP_R^{i+1}$ along 
$\simt(P_\tot(\unl{R}_l \cap \Dom(\calP_R^i)))$ and $\simt(P_\tot(\unl{R}_l \cap \Dom(\calP_R^{i+1})))$ respectively.
 \item For each $\una{R}_l$ intersecting $\calP_R^{i}$, the intersection $\una{R}_l \cap \bdy\Dom(P_R^i)$ maps to two $\alpha$ segments $d^-(\una{R}_l),d^+(\una{R}_l) \in \{d_{y_{k}}\}$.
The corresponding $\alpha$ segments $t^-(\una{R}_l),t^+(\una{R}_l) \in \{t_{y_{k}}\}$, 
together with the two components of $\simt(P_\tot(\una{R}_l \cap \Dom(\calP_R^i)))$, 
define a rectangular region $\una{A}_l$ of $T(\una{R}_l)$. We glue $\una{A}_l$ to $\calP_R^i$ along $\simt(P_\tot(\una{R}_l \cap \Dom(\calP_R^i)))$.
 \item The $\alpha$ segment $t_{y_c}$, together with $\beta(y_c)$ and $\simt(P_\tot(R_\init \cap \Dom(\calP_R^1))) \cap \calD(\psi_0)$,
define a rectangular region of $T_{w_b} \cap \calD(\psi_0)$,
which we glue to $\calP_R^1$ along $\simt(P_\tot(R_\init \cap \Dom(\calP_R^1))) \cap \calD(\psi_0)$.

Similarly, the $\alpha$ segment $t_{y_{c'}}$, together with $\beta(y_{c'})$ and $\simt(P_\tot(R_\term \cap \Dom(\calP_R^r))) \cap \calD(\psi_0)$,
define a rectangular region of $T_{w_{b'}} \cap \calD(\psi_0)$,
which we glue to $\calP_R^r$ along $\simt(P_\tot(R_\term \cap \Dom(\calP_R^r))) \cap \calD(\psi_0)$.
\end{itemize}

Finally, we let $P_M^{i,j}$ be the component of $\calD(\mathbb{I}) \setminus (\boldsymbol{\alpha}\cup\boldsymbol{\beta})$ containing $\calP_M^{i,j}$,
and this completes the proof.
\end{proof}

\begin{example}
Continuing Example \ref{GluingConstructionExample}, there is a relation in $M(\phi) \otimes (\dd \otimes M(\psi))$ (see Figure \ref{GluingConstructionExampleFig})
coming from a diagram as in Figure \ref{ExampleSpiderRelation}.
Figure \ref{PolygonConstruction1} illustrates how the corresponding polygons in $\calD(\phi_0),\calD(\mathbb{I})$, and $\calD(\psi_0)$,
together with regions of $D_{w_1}$ and $D_{w_2}$, are glued together as in \textbf{Direction 1}
to form the corresponding polygon in $\tilde{\calD}(\psi_0\circ\phi_0)$.
Figure \ref{PolygonConstruction2} illustrates how \textbf{Direction 2} recovers the polygons of Figure \ref{ExampleSpiderRelation} from a polygon
in $\tilde{\calD}(\psi_0\circ\phi_0)$.
\end{example}

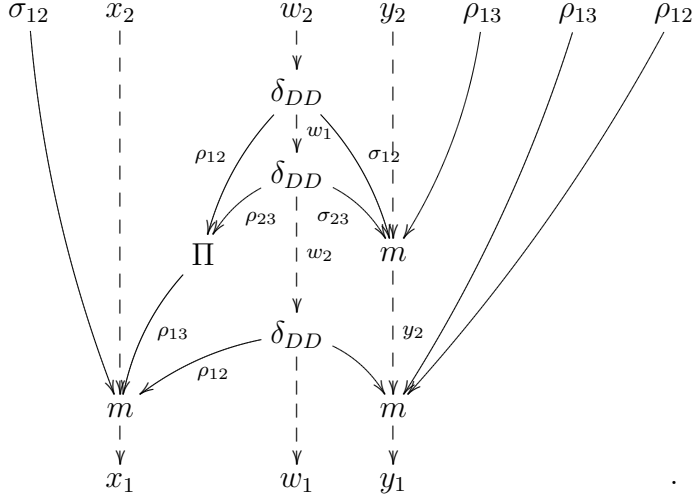
\begin{figure}
\begin{center}
\begin{align*}
\xymatrix@C=1.2em@R=1.2em
{
\sigma_{12}\ar@/_/[dddddr]&x_2\ar@{-->}[ddddd]&&w_2\ar@{-->}[d]&y_2\ar@{-->}[ddd]&\rho_{13}\ar@/^/[dddl]&\rho_{13}\ar@/^/[dddddll]&\rho_{12}\ar@/^/[dddddlll]\\
&&&\delta_{DD}\ar_{\rho_{12}}@/_/[ddl]\ar^{w_1}@{-->}[d]\ar^{\sigma_{12}}@/^/[ddr]&&&&\\
&&&\delta_{DD}\ar^{w_2}@{-->}[dd]\ar^{\rho_{23}}@/_/[dl]\ar_{\sigma_{23}}@/^/[dr]&&&&\\
&&\Pi\ar^{\rho_{13}}@/_/[ddl]&&m\ar^{y_2}@{-->}[dd]&&&\\
&&&\delta_{DD}\ar@{-->}[dd]\ar^{\rho_{12}}@/_/[dll]\ar@/^/[dr]&&&&\\
&m\ar@{-->}[d]&&&m\ar@{-->}[d]&&&&\\
&x_1&&w_1&y_1&&&.\\
} 
\end{align*}
\end{center}
\caption{A relation in $M(\phi)\otimes (\dd \otimes M(\psi))$.}
\label{ExampleSpiderRelation}
\end{figure}
\begin{remark}
Note that our choice of the parenthesization $M(\phi) \otimes (\dd \otimes M(\psi))$ (rather than $(M(\phi) \otimes \dd)\otimes M(\psi)$) manifests itself in $\textbf{Direction 1}$
and $\textbf{Direction 2}$. Namely, in $\textbf{Direction 1}$ we have a single polygon $P_L$ on the left and multiple polygons $P_R^i$ on the right, instead of vice versa,
and this results in our gluing the left ends of $P_M^{i,j}$ and $P_M^{i,j+1}$ directly adjacent to each other on the right side of $P_L$, while
the right sides of $P_M^{i,j}$ and $P_M^{i,j+1}$ are separated by a $\beta$ arc on the left side of $P_M^i$.
In $\textbf{Direction 2}$, we replace each left domain $\ovl{R}_l$ with a single domain $\ovl{A}_l$ in order to construct a single polygon $P_L$ on the left, whereas
we replace each right domain $\unl{R}_l$ with two domains in order to split the corresponding region on the right side into multiple polygons.
\end{remark}
\begin{figure}
\begin{center}
\includegraphics[scale=.5]{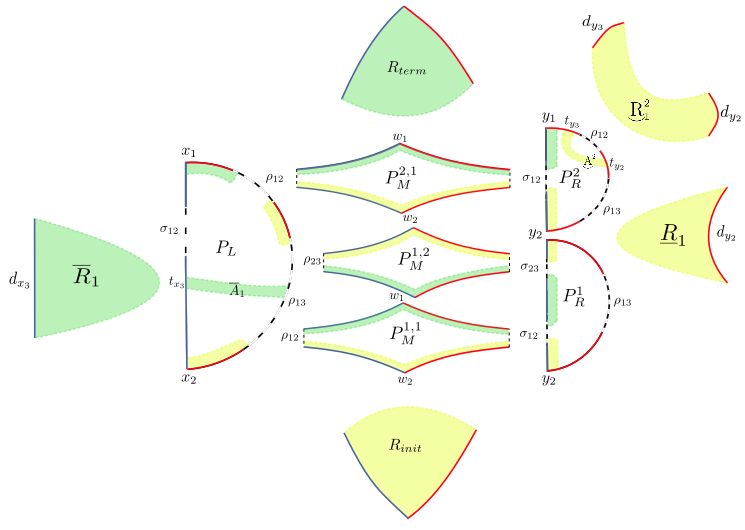}
\end{center}
\caption{\textbf{Direction 1} removes the preimage of $T_{w_1}$ and $T_{w_2}$ from $P_L,P_M^{1,1},P_M^{1,2},P_M^{2,2},P_R^1$, and $P_R^2$ and glues in the pieces $\ovl{R}_1,\unl{R}_1,\una{R}_1^2,
R_\init$, and $R_\term$ to form a polygon in $\tilde{\calD}(\psi_0\circ\phi_0)$.}
\label{PolygonConstruction1}
\end{figure}
\begin{figure}
\begin{center}
 \includegraphics[scale=.5]{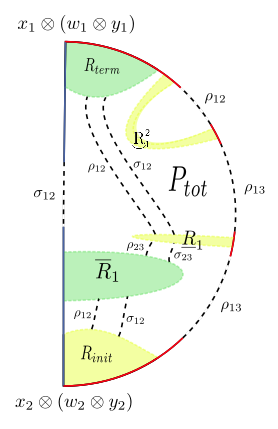}
\end{center}
\caption{\textbf{Direction 2} performs the reverse process of \textbf{Direction 1} to recover $P_L,P_M^{1,1},P_M^{1,2},P_M^{2,2},P_R^1$, and $P_R^2$, and hence the relation
of Figure \ref{ExampleSpiderRelation}1, from a polygon in $\tilde{\calD}(\psi_0\circ\phi_0)$.}
\label{PolygonConstruction2}
\end{figure}

\subsection{Completing the proof}\label{Completing the proof}
In this section we prove two theorems. The first shows that the identity mapping class gives the identity module, and the second shows that no other mapping class gives the identity module.
\begin{definition}
For an $\calA_{\infty}$-algebra $\calA$ over $\kk$, the {\em identity bimodule $\aoa$} is given as follows.
As a $\kk$-module $\aoa$ is isomorphic to $\kk$. For $k \neq 2$, $\delta_{DA}^{1,k} = 0$, while
\begin{align*}
\delta_{DA}^{1,2}(\iota,a) = a \otimes \iota, 
\end{align*}
where $\iota$ is the generator of $\aoa$.
\end{definition}

\begin{lemma}\label{possibleshortchords}
For a chord $\rho \in \calB(\calZ)$, let $|\rho|$ denote the number of short chords composing $\rho$, and similarly for chords in $\calB(\calZ')$.
Let $P$ be a polygon in $\calD(\mathbb{I})$ through $(\sigma_1,...,\sigma_m)$ and $(\rho_1,...,\rho_n)$. We have
\begin{align}\label{countingsigmashortchords}
\left(\sum_{i=1}^m|\sigma_i|\right) - m = n - 1
\end{align}
and
\begin{align}\label{countingrhoshortchords}
\left(\sum_{i=1}^n|\rho_i|\right) - n = m - 1.
\end{align}
\end{lemma}
\begin{proof}
Recall that in $\calD(\mathbb{I})$, each $\alpha$ arc intersects intersects its dual $\beta$ arc exactly once and is disjoint from every other $\beta$ arc.
It follows that each component of $C$ of $P^{-1}(\boldsymbol{\beta})$ must have endpoints on both sides of $\Dom(P)$.
To see this, note that if both endpoints of $C$ were on the right side of $\Dom(P)$, we could cut $P$ along $C$ to obtain a polygon in $\calD(\mathbb{I})$ passing through no chords
on the left side, which violates the first sentence of this proof.
On the other hand, $C$ cannot have both endpoints on the left side of $\Dom(P)$ as this would violate the conditions in the first paragraph of
Section \ref{arc diagrams and the algebra B(Z)}.

Note also that the left endpoint of $C$ lies at an intersection point of two adjacent $\calB(\calZ')$ short chord preimages,
whereas the right side is the unique intersection point of the $\alpha$ arc connecting $\rho_i$ and $\rho_{i+1}$ for some $1 \leq i < n$.
Then (\ref{countingsigmashortchords}) follows by noting that the number of left endpoints of components of $P^{-1}(\boldsymbol{\beta})$ is $\left(\sum_{i=1}^m|\sigma_i|\right) - m$, 
whereas the number of right endpoints is $n-1$.

(\ref{countingrhoshortchords}) is proved similarly.
\end{proof}

\begin{theorem}\label{identityisidentity}
$\dd \boxtimes M(\mathbb{I})$ is isomorphic as a type DA bimodule to $\aoa$.
\end{theorem}
\begin{proof}
First observe that $\dd \boxtimes M(\mathbb{I})$ has one generator $w_i \otimes x_i$ for each idempotent $I_i$ of $\calB(\calZ)$,
and hence its generators are in a natural bijective correspondence with the generators of $\aoa$.
The structure map $\delta_{DA}^1$ of $\dd \boxtimes M(\mathbb{I})$ is given by 
\begin{align}\label{twopartdiagram}
\xymatrix@=1em
{
&w_i\ar@{-->}[d]&x_i\ar@{-->}[dddddd]&\rho_{i_1,i_2}\ar@/^/[ddddddl]&...\ar@/^/@{=>}[ddddddll]&\rho_{i_{k-1},i_k}\ar@/^/[ddddddlll]\\
&\bdy\ar@{-->}[d]\ar@/_/[dddddl]\ar@/^/[dddddr]&&&&\\
&\ar@{.}[dd]&&&&\\
&\ar@/_/@{=>}[dddl]\ar@/^/@{=>}[dddr]&&&&\\
&\ar@{-->}[d]&&&&\\
&\bdy\ar@/_/[dl]\ar@/^/[dr]\ar@{-->}[dd]&&&&\\
\Pi\ar[d]&&m\ar@{-->}[d]&&&\\
&&&&&\\
} 
\end{align}
for $\rho_{i_1,i_2},...,\rho_{i_{k-1},i_k} \in \calB(\calZ)$.
Since the left inputs of the $m$ in the diagram are short chords, we must have $k = 2$ by (\ref{countingsigmashortchords}) of Lemma \ref{possibleshortchords}.
Thus it suffices to show that there is a unique polygon $P: \D^2 \rightarrow \calD(\mathbb{I})$ with initial point $x_i$ passing through entirely short chords on the left
and $\rho_{i_1,i_2}$ on the right, and that $P$ fits into a diagram of the form (\ref{twopartdiagram}).

Let $\rho_{i_1,i_1+1},\rho_{i_1+1,i_1+2},...,\rho_{i_2-1,i_2} \in \calB(\calZ)$ be short chords such that
\begin{align*}
\rho_{i_1,i_2} = \rho_{i_1,i_1+1}\rho_{i_1+1,i_1+2}...\rho_{i_2-1,i_2}. 
\end{align*}
For $0 \leq j < i_2-i_1$, let $P_{i_1+j,i_1+j+1}$ denote the component of $\calD(\mathbb{I}) \setminus (\boldsymbol{\alpha}\cup\boldsymbol{\beta})$ defined by 
$\rho_{i_1+j,i_1+j+1}$ and the corresponding short chord $\sigma_{i_1+j,i_1+j+1}$.
Observe that $P_{i_1+j,i_1+j+1}$ is a hexagon with one side corresponding to $\rho_{i_1+j,i_1+j+1}$, one side corresponding
to $\sigma_{i_1+j,i_1+j+1}$, two $\alpha$ sides $\alpha^\init_{i_1+j,i_1+j+1}$ and $\alpha^\term_{i_1+j,i_1+j+1}$ intersecting the initial and terminal points of $\rho_{i_1+j,i_1+j+1}$ respectively,
and two $\beta$ sides
$\beta^\init_{i_1+j,i_1+j+1}$ and $\beta^\term_{i_1+j,i_1+j+1}$ 
intersecting the inital and terminal points of $\sigma_{i_1+j,i_1+j+1}$ respectively.
We proceed by setting
\begin{align*}
P := P_{i_1,i_1+1} \cup P_{i_1+1,i_1+2} \cup ... \cup P_{i_2-1,i_2}, 
\end{align*}
where $P_{i_1 + j,i_1+ j + 1}$ is glued to $P_{i_1 + j + 1, i_1 + j + 2}$ along 
$\alpha^\term_{i_1+j,i_1+j+1}$ and $\alpha^\init_{i_1+j+1,i_1+j+2}$.
The reader can easily check that $P$ is a polygon of the desired form.

As for uniqueness, suppose $P': \D^2 \rightarrow \calD(\mathbb{I})$ is a polygon with initial point $x_i$ passing through entirely short chords on the left and $\rho_{i_1,i_2}$ on the right.
As in the proof of Lemma \ref{possibleshortchords}, there must be pairwise disjoint embedded paths $\alpha_{i_1+j,i_1+j+1}: [0,1] \rightarrow \Dom(P')$ 
for $0 \leq j < i_2 - i_1 - 1$ 
with $P'(\alpha_{i_1+j,i_1+j+1}) \subset \boldsymbol{\alpha}$,
$P'(\alpha_{i_1+j,i_1+j+1}(0)) = \rho^+_{i_1+j,i_1+j+1}$
and $\alpha_{i_1+j,i_1+j+1}(1)$ lying on the left side of $\Dom(P')$. These $\{\alpha_{i_1+j,i_1+j+1}\}$ divide $\Dom(P')$ into regions, the restriction of $P'$ to which are diffeomorphisms
$P'_{i_1,i_1+1},...,P'_{i_2-1,i_2}$ onto $P_{i_1,i_1+1},...,P_{i_2-1,i_2}$ respectively. 
It then follows that $P'$ differs from $P$ by some diffeomorphism $\D^2 \rightarrow \D^2$, i.e. $P'$ and $P$ are equivalent.
\end{proof}

\begin{theorem}
If $\dd \boxtimes M(\phi_0)$ is quasi-isomorphic to $\aoa$, then $\phi_0$ is isotopic to $\mathbb{I}$.
\end{theorem}
\begin{proof}
Suppose $\dd \boxtimes M(\phi_0)$ is quasi-isomorphic to $\aoa$. Then by Corollary \ref{bigcorollary} and Theorem \ref{identityisidentity}
we have $M(\phi_0) \simeq M(\mathbb{I})$. 
Let $I_i \in \calB(\calZ)$ and $J_i \in \calB(\calZ')$ be idempotents corresponding to $\alpha_i$ and $\beta_i$ respectively.
Observe that $J_j H_*(M(\phi_0)) I_i$ is the Floer homology $HF(\alpha_i,\beta_j)$ of $\alpha_i$ and $\beta_j$, i.e. the homology of the chain complex 
generated by intersection points of $\alpha_i$ and $\beta_j$ whose differential counts (equivalence classes of) immersed bigons between $\alpha_i$ and $\beta_j$.
Since $HF(\alpha_i,\beta_i)$ is an isotopy invariant of $\alpha_i$ and $\beta_j$, we can assume there are no bigons between $\alpha$ and $\beta$, and we have
\begin{align*}
 J_j H_*(M(\phi_0)) I_i = \mathit{i}(\alpha_i,\beta_j),
\end{align*}
where $\mathit{i}(\alpha_i,\beta_j)$ is the geometric intersection number of $\alpha_i$ and $\beta_j$, i.e. the minimal 
number of intersection points over all isotopic representatives of $\alpha_i$ and $\beta_j$.

But then since $J_j H^*(M(\phi_0)) I_i$ is a quasi-isomorphism invariant,
we must have $\mathit{i}(\alpha_i,\beta_j) = \delta_{ij}$. It follows that, up to isotopy,
$\phi_0$ fixes the dual curves $\{\eta_i\}$. Since $F \setminus \cup_i \eta_i$ is a collection of disks,
$\phi_0$ must be isotopic to the identity.
\end{proof}


\end{document}